\documentclass[11pt,amsfonts]{article}
\usepackage[utf8]{inputenc}
\usepackage[T1]{fontenc} 
\usepackage{tikz}
\usepackage[english]{babel} 
\usepackage{mathrsfs}
\usepackage{amssymb}
\usepackage[cmex10]{amsmath}
\usepackage{amsthm}
\usepackage{wrapfig}
\usepackage{latexsym}
\usepackage{amsfonts}
\usepackage{color}
\usepackage[font=small, labelfont=bf]{caption}
\usepackage{ctable}
\usepackage{floatrow}
\usepackage[colorlinks=true, citecolor=red, linkcolor=blue]{hyperref}
\usepackage{graphicx}
\usepackage{subfig}
\usepackage{verbatim}
\usepackage{epstopdf}
\usepackage{float}

\usepackage[top=2cm , bottom=2cm , left=2.5cm ,right=2.5cm ]{geometry}

\usepackage{cases}

\theoremstyle{plain}
\newtheorem{definition}{Definition}
\newtheorem{theorem}{Theorem}

\newtheorem{proposition}{Proposition}

\newtheorem{lemma}{Lemma}

\newtheorem{remark}{Remark}

\numberwithin{equation}{section}
\numberwithin{theorem}{section}
\numberwithin{proposition}{section}
\numberwithin{definition}{section}
\numberwithin{remark}{section}
\numberwithin{corollary}{section}
\numberwithin{lemma}{section}

\usepackage{graphicx}

\makeatletter
\newcommand*\bigcdot{\mathpalette\bigcdot@{.5}}
\newcommand*\bigcdot@[2]{\mathbin{\vcenter{\hbox{\scalebox{#2}{$\m@th#1\bullet$}}}}}
\makeatother
\begin{document}
	
	\renewcommand{\thefootnote}{\arabic{footnote}}
	%$^{,}$
	\begin{center}
		{\Large \textbf{ Stochastic differential equations  with respect to optional semimartingales and two reflecting regulated barriers  }} \\[0pt]
		~\\[0pt] Astrid Hilbert\footnote[1]{Mathematics Department, Linnaeus University, Sweden.E-mail: \texttt{astrid.hilbert@lnu.se}}, Imane Jarni\footnote[2]{Mathematics Department, Faculty of Sciences Semalalia, Cadi Ayyad University, Morocco. 
			E-mail: \texttt{jarni.imane@gmail.com, ouknine@uca.ac.ma}} and Youssef Ouknine \textcolor{blue}{\footnotemark[2]}\footnote[3]{Africa Business School, Mohammed VI Polytechnic University, Morocco
			E-mail: \texttt{youssef.ouknine@um6p.ma, }} \footnote[4]{Hassan II Academy of Sciences and Technologies, Rabat, Morocco.}
		\\[0pt]

		~\\[0pt]
	\end{center}

	\renewcommand{\thefootnote}{\arabic{footnote}}

	\begin{abstract}
	In this work, we introduce a new   Skorokhod problem with two reflecting barriers when the trajectories of the driven process  and the barriers  are  right and left limited. We show that this problem has an explicit unique solution in a deterministic case.  Then, we apply our result to study the existence and uniqueness of solutions of reflected stochastic differential equations with respect to optional semimartingales. The study is carried out on a probability space that does not necessarily satisfy  the usual conditions.
		
	\end{abstract}
	
\section{Introduction}

In this paper we consider reflected stochastic differential equations of this form:
\begin{equation}
X_{t}=X_{0}+\sigma(.,X)\bigcdot M_{t}+ b(.,X)\bigcdot V_{t}+K_{t}, \;\;t\geq 0,\label{Equ}
\end{equation}
where $M$ is an optional  local martingale, $V$ is an adapted  process of bounded variation. For  given processes $L$ and $U$ with regulated trajectories  such that $L_t\leq U_t$ for $t\geq 0$, $X$ is the reflecting process on the time-dependent interval $[L,U]$, i.e., $L_t\leq X_t\leq U_t$ for all $t\geq 0$, $K$ is a  process of bounded variation which compensates the reflection of $X$,   $\sigma(.,X)\bigcdot M$ and $b(.,X)\bigcdot V$ are the integrals with respect to optional semimartingales ( see Section \ref{sec2} for the definition of such an integral and Section \ref{sec4} for a precise definition of the equation \ref{Equ}).

 	One of the goals of this  work is to study existence and uniqueness of solutions of  stochastic differential equations with two reflecting barriers in a probability  space that does not necessarily satisfy the usual hypothesis of right continuity and completeness of the filtration (\textit{un} usual probability space). Since optional martingales in such a probability space have trajectories with right and left limits and they are not necessarily right  continuous left limited, we  first need to define a suitable reflection problem when the trajectories of the driver and the barriers are right and left limited.  Before explaining  precisely our method and  results, let us recall some  works  concern reflected stochastic differential equations and Skorokhod map:  Reflected SDEs was firstly introduced and dealt by Skorokhod  \cite{skorokhod} in the case where $M$ is a Brownian motion, $\rm{d}V$ is the Lebesgue measure, the lower barrier $L=0$, and the upper barrier $U=\infty$. In this setting, many works have been done to extend  this type of equations for an important class of driving processes or reflecting domains and to study the properties of Skorokhod map  (see, e.g. \cite{el karoui}, \cite{el karoui1}, \cite{burzdy}, \cite{kruk}, \cite{protter1},  \cite{slominski1},\cite{slominski2},\cite{slominski3}, \cite{tanaka}). In these papers, the study of reflected equations is made in the framework where the driven processes and the barriers are continuous or  right continuous left limited. Moreover, the probability space associated is assumed to  satisfy the usual conditions.
  However, probability spaces which are not necessarily satisfying the usual conditions exist. Examples of such probability spaces are given in \cite{fleming} and  \cite{karatzas}. 
 	
 	Stochastic calculus in a general probability space was  firstly studied  by Dellacherie and  Meyer in \cite{dellacherie}, then Gal’{\v c}huk in  \cite{galchuk},\cite{galchuk1},\cite{galchuk82}, and \cite{galchuk85} has developed many results in this context. 
 	
 	Up to our knowledge, there are few works deal with non reflected equation of type \eqref{Equ}. We quote the work of  Lenglart \cite{lenglart},  Gasparyan  \cite{gasparyan}, and the  recent  works of Abdelghani and Melnikov  \cite{abdelghani},\cite{Abdelghani},\cite{Abdelghani1}, and \cite{Abdelghani2} in which the authors use these equations to model financial market with a filtration  that does not necessarily satisfy usual conditions.

 Regarding reflected SDEs driven by optional semimartingales, we quote the work of Jarni and  Ouknine  \cite{jarni}. In this work, the authors have extended the  reflection problem with jump when the driven process is right continuous left limited, introduced firstly by Chaleyat-Maurel, El Karoui and Marchal  in \cite{el karoui1}, to the case when the  driven process has right and left limited trajectories. The  reflection problem  considered in this work is  with one reflecting barrier. Note that the type of reflection in \cite{el karoui1}  the jumps are reflected in the boundary unlike the classical reflection problem dealt by Tanaka \cite{tanaka}, the jumps are absorbed at the boundary. An extension of the reflection problem introduced in \cite{tanaka} for regulated input and one reflecting barrier is studied in \cite{hilbert}. 

 In this work, we are concerned  with  equations with the classical Skorokhod problem with two reflecting barriers when the trajectory of the driven process and the  barriers is only  right and left limited.  In the all mentioned above papers, concern reflected SDEs, the solution is always depending on the paths of $X$, and its construction  is based on the properties of deterministic mapping on the space of continuous, or right continuous left limited or right and left limited functions. Our first objective in this paper is to define and solve a reflection problem in deterministic case  that covers  regulated functions. This definition  coincides with the classical reflection  problem in the case when the driven and the barriers are right continuous  left limited, and it is based on the reflection problem introduced in \cite{hilbert}. We show that this problem has an explicit unique solution. We  also show that the Skorokhod map on the space of right and left limited functions is a functional of the Skorokhod map with one barrier. The second purpose of this work is to study the existence and uniqueness of solutions of equation \eqref{Equ} in an
 \textit{un} usual probability space.
	
	The paper is organized as follows. Section \ref{sec2} is devoted for preliminaries. In the section \ref{sec3}, we introduce a deterministic Skorokhod problem with two reflecting barriers, we show that this problem has an explicit unique solution.
	In the section \ref{sec4}, we define and study  the reflection problem for processes, then we study existence and uniqueness of solutions of reflected stochastic differential equations driven by optional semimartingales in an \textit{un} usual  probability space using fixed point argument. 
\section{Preliminaries}\label{sec2}
	We say that a function $y :\mathbb{R}^{+} \to \mathbb{R}$ is \emph{regulated}  if $y$ has a left limit in each point of $\mathbb{R}^{*,+}$, and has a right limit in each point of $\mathbb{R}^{+}$. We denote by $\mathcal{R}(\mathbb{R}^{+},\mathbb{R})$ the set of regulated functions from $\mathbb{R}^{+}$ to $\mathbb{R}$, where $\mathbb{R}^{+}$ denotes the interval $[0,+\infty[$ and  $\mathbb{R}^{*,+}$ denotes the interval $]0,+\infty[$.

For  $y\in \mathcal{R}(\mathbb{R}^{+},\mathbb{R})$, and $t>0$ we denote by  $y_{t^{-}}:=\lim\limits_{s\nearrow t} y_{s}$ its left limit, and $y_{t^{+}}:=\lim\limits_{s\searrow t} y_{s}$ its right limit. We set  $\Delta^{+} y_{t}:=y_{t^{+}}-y_{t}$, and $\Delta^{-} y_{t}:=y_{t}-y_{t^{-}}.$

We consider a complete probability space $(\Omega,\mathcal{F},\mathbb{P})$ equipped with a filtration $\mathbb{F}=(\mathcal{F}_{t})_{t\geq 0}$, which is not assumed to be right or left continuous and it is not assumed to be completed.
 We  introduce the following families of $\sigma$- algebras: $$\mathbb{F}_{+}=(\mathcal{F}_{t_{+}})_{t\geq 0},\;\;\;
\mathbb{F}^{\mathbb{P}}=(\mathcal{F}_{t}^{\mathbb{P}})_{t\geq 0},\;\;\;
\mathcal{F}_{+}^{\mathbb{P}}=(\mathcal{F}_{t_{+}}^{\mathbb{P}})_{t\geq 0},$$
where $\mathcal{F}_{t_{+}}={\displaystyle\bigcap}_{t<s}\mathcal{F}_{s}$.  $\mathbb{F}^{\mathbb{P}}$ and $\mathbb{F}_{+}^{\mathbb{P}}$ are the completion of $\mathbb{F}$ and  $\mathbb{F}_{+}$  under $\mathbb{P}$, respectively.

If $\mathbb{G}=(G_t)_{t\geq 0}$, is a nondeacresing family of $\sigma$-algebras, 
$\mathcal{O}(\mathbb{G})$ and $\mathcal{P}(\mathbb{G})$ denotes the $\sigma$- algebras on $\Omega\times\mathbb{R}^{+}$, with  $\mathcal{O}(\mathbb{G})$ is generated by all $\mathbb{G}$-adapted, right continuous and left limited processes and  $\mathcal{P}(\mathbb{G})$ is  generated by all $\mathbb{G}$- adapted, left continuous and right limited processes.\\
$\mathcal{O}(\mathbb{F})$- or $\mathcal{P}(\mathbb{F})$ measurable processes are said to be optional or predictable respectively.

A process $M$ is said to be  an optional martingale, and we write $M\in\mathcal{M}$, if $M$ is an optional process and there exists  an $\mathcal{F}$-measurable random variable $\hat{M}$ such that $\mathbb{E}(|\hat{M}|)<+\infty$ and $M_{\tau}=\mathbb{E}[\hat{M} | \mathcal{F}_{\tau}]$ a.s on the set $\{\tau<\infty\}$ for every stopping time $\tau$.
We denote by $\mathcal{M}^{2}$ the set of optional martingale such that $\mathbb{E}(|\hat{M}|^{2})<+\infty$.

A process $M$ is said to be an optional local martingale if there exists a sequence $(R_{n},M^{n})$, $n\in\mathbb{N}$, where $R_{n}$ is a  sequence of stopping times in wide sens, and $M^{n}$ is an optional martingale, $R_{n}\uparrow\infty$ a.s, $M=M^{n}$ on $[0,R_{n}]$ and the random variable $M_{R_{n}^{+}}$ is integrable for any $n\in\mathbb{N}$.

From Theorem 4.10 in \cite{galchuk}, an optional local martingale  $M$  has the following decomposition; $M=M^c+M^d+M^g$, where $M^c$ is a continuous local martingale, $M^d$ is right continuous left limited local martingale and $M^g$ is left continuous right limited local martingale. We denote by $M^r$ the process $M^r:=M^c+M^d$.

Let $V$ be a process of bounded variation, then $V$ had the following decomposition:
$
V_{t}=V^{c}_{t}+\displaystyle\sum_{0<s\leq t} \Delta^{-} V_{s}+\displaystyle\sum_{0 \leq s<t} \Delta^{+} V_{s},
$
where $V^{c}$ is a continuous process of bounded variation. 
We denote by $V_{t}^{r}:=V_{t}^{c}+V_{t}^{d}=V^{c}_{t}+\displaystyle\sum_{0<s\leq t} \Delta^{-} V_{s}$,  by $V_{t}^{g}=\displaystyle\sum_{0\leq s<t} \Delta^{+} V_{s}$, and by $|V|_{t}$ the total variation of the process $V$ on the interval $[0, t]$.

Let $X$ be an optional semimartingale with the following decomposition: $X=M+V$, where $M$ is an optional local martingale, and $V$ is an optional process of bounded variation.
$H\bigcdot X$ denotes the following stochastic integral:
$$H\bigcdot X_t=\int_{]0,t]} H_{s^-} \rm{d}M_s^r+\int_{[0,t[} H_{s^-} \rm{d}M_{s^+}^g+\int_{]0,t]} H_{s^-} \rm{d}V_s^r+\int_{[0,t[} H_{s^-} \rm{d}V_{s^+}^g.$$
$\displaystyle\int_{]0,t]}H_{s^{-}} dM_{s}^{r}$ is the usual integral with respect to right continuous left limited local martingale and $\displaystyle\int_{[0,t[} H_{s^-}\rm{d}M^{g}_{s^{+}}$ is Gal’{\v c}huk stochastic integral with respect to left continuous and right limited local martingale (see \cite{galchuk}). $\displaystyle\int_{]0,t]}H_{s^{-}} dV_{s}^{r}$ and $\displaystyle\int_{[0,t[} H_{s^-}\rm{d}V^{g}_{s^{+}}$ are in Stieltjes sens.
$V^r$ and  $V^g$ denote the right continuous part of $V$ and left continuous part of $V$, respectively.

 we denote by  $a \land b=\min(a,b)$, and   by $a \lor b=\max(a,b)$.

In what follows, we assume that all the regulated  processes
considered do not jump at $0$.

\section{Skorokhod problem with two reflecting barriers}\label{sec3}
We motivate  the definition of the reflection Skorokhod  problem with  two reflecting barriers by the following example.

Let $t_0>0$  and $p,q\in\mathbb{R}$. we consider $0$ and $a$ ($a>0$)  as the lower and the upper barrier, respectively . We consider the regulated function  $y$ defined by:
  $$y_{t}=p1_{\{t=t_{0}\}}+q 1_{\{t>t_{0}\}},\;\;\text{for}\;t\geq 0.$$
 We know that  a solution $(x,k)$ of  Skorokhod problem associated with $y$ and the barriers $0$ and $a$  should satisfy:
 $$x_t=y_t+k_t,\;\;0\leq x_t\leq a\;\;\text{for every} \;t\geq 0,\;\;\text{and} \;\;k\; \text{has a bounded variation}.$$
 We have $$y_{t_0}=p\;\text{ and }\;\;y_{t^+_0}=q.$$
If $p<0$ and $q>a$, then $k$ should react in order to push $y$. In this case,  it is natural to have  $$x_{t_0}=0\;\;\text{and}\;\;x_{t_0^+}=a.$$ 
Therefore, we get
 $$\Delta^+k_{t_0}=\Delta^+x_{t_0}-\Delta^+y_{t_0}=a-q+p<0,$$
 it follows that $k$ has a jump on the right at $t_0$.
 
 If  $p\geq 0$ or  $q\leq a$. We assume for example $p< 0$, then $y_{t_0}<0$ and $y_{t_0^+}\leq a$. 
 A natural solution of this reflection problem should be constructed  by  alternating between the two standard one sided reflecting operators which are  given  in Proposition \ref{r1} and \ref{r2}.
 In this case, the expression of 
 $k_{t_0}$  will depend on the solution of the reflection problem with one barrier. From Proposition \ref{r1}, we have
 $k_{t_0}=k_{t_0^+}=-(y_{t_0}\land y_{{t^+_0}})$ but at the same time we should not allow $k_{t_0}$ to  exceeds $a-y_{t_0}$, because otherwise, we may have many oscillations of  $k$  at $t_0$ which means that $k$ may have an infinite variation at $t_0$. Thus, a natural value of $k$ at $t_0$ is
 \begin{equation*}
 k_{t_0}=-\big((y_{t_0}\land y_{t_0^+})\lor(y_{t_0}-a)\big),\label{v1}
 \end{equation*}and
 \begin{equation*}
 k_{t_0^+}=-\big((y_{t_0}\land y_{t_0^+})\lor(y_{t_0^+}-a)\big).\label{v2}
 \end{equation*}
 We compute the right jump of $k$ at $t_0$, we obtain:
 \begin{center}
 	If $\Delta^+k_{t_0}>0$,  $x_{t_0}=a$ and $x_{t^+_0}=0$,  
 \end{center}
 and
 \begin{center}
 	If $\Delta^+k_{t_0}<0$, $x_{t_0}=0$ and $x_{t^+_0}=a$.
 \end{center}
 Inspired by this example, we define the reflection problem with two reflecting barriers as follows.
\begin{definition}\label{def1}
 Let  $y$, $l$ and $u$ be in  $\mathcal{R}(\mathbb{R}^{+},\mathbb{R})$  such that $l_0\leq y_{0}\leq u_0$. We say that a pair of functions  $(x,k)$ in $\mathcal{R}(\mathbb{R}^+,\mathbb{R})\times \mathcal{R}(\mathbb{R}^+,\mathbb{R})$  is a solution of the  reflection problem associated with $y$ and the barriers $l$ and $u$ if:
	\begin{itemize}
		\item[(i)]$x=y+k=y+\phi^1-\phi^2$;
		\item[(ii)] $l\leq x \leq u$;
		\item[(iii)] $\phi^1$ and  $\phi^2$ are non-decreasing,  with $\phi^1_{0}=\phi^2_{0}=0$;
		\item[(iv)]
		\begin{equation*}
		\int_{[0,+\infty[} \big((x_{s}-l_s)\land (x_{s^{+}}-l_{s^+})\big)\rm{d}\phi^{1,r}_{s}=\int_{[0,+\infty[} \big((u_s-x_{s})\land (u_{s^+}-x_{s^{+}})\big)\rm{d}\phi^{2,r}_{s}=0.
		\end{equation*}
\item[(v)] 	For every $t\in\mathbb{R}^{+}$,
				\begin{equation}
			\sum_{s\leq t}(x_{s^{+}}-l_{s^+})\Delta^{+}\phi^1_s= \sum_{s\leq t}(u_{s}-x_{s})\Delta^{+}\phi^1_s=0. \label{d1}
			\end{equation}
			\begin{equation}
			\sum_{s\leq t}(u_{s^+}-x_{s^{+}})\Delta^{+}\phi^2_s= \sum_{s\leq t}(x_{s}-l_{s})\Delta^{+}\phi^2_s=0.\label{d2}
			\end{equation}
	\end{itemize}
 This problem will be denoted by $RP^u_l(y)$.
\end{definition}

\begin{remark}
	\begin{itemize}
\item[i)] The statement (v) in Definition \ref{def1} is equivalent to the following:
\begin{center}
$x_t=u_t$ and $x_{t^+}=l_{t^+}$ for every $t$ such that $\Delta^+k_t>0$,
\end{center}
and
\begin{center}
$x_t=l_t$ and $x_{t^+}=u_{t^+}$ for every $t$ such that $\Delta^+k_t>0.$
\end{center}
\item[ii)] Let $(x,k)$ be a solution of the reflection problem $RP()$. By a slight generalization of the result in \cite{kela}, one can show that the statements
\begin{equation*}
\int_{[0,+\infty[} \big((x_{s}-l_s)\land (x_{s^{+}}-l_{s^+})\big)\rm{d}\phi^{1,r}_{s}=0\;\text{and}\; \int_{[0,+\infty[} \big((u_s-x_{s})\land (u_{s^+}-x_{s^{+}})\big)\rm{d}\phi^{2,r}_{s}=0
\end{equation*}
in Definition \ref{def1} are equivalent to the following:
$$ \text{ For every t such that } \Delta^+\phi^1_t=0\;\text{ and } \;\phi^{1}_{t^-}< \phi^{1}_s \;\text{for all}\; s> t, \; (x_{t}-l_t)\land (x_{t^{+}}-l_{t^+})=0$$
and
$$\text{ For every t such that } \Delta^+\phi^2_t=0\;\text{ and }\;\phi^{2}_{t^-}< \phi^{2}_s \;\text{for all}\; s> t, \;\; (u_t-x_t)\land (u_{t^+}-x_{t^{+}})=0,$$
respectively.
	\end{itemize}
\end{remark}

\begin{lemma}\label{lem1}
Let  $y$, $l$ and $u$ be in  $\mathcal{R}(\mathbb{R}^{+},\mathbb{R})$  such that $l\leq u$ and $l_0\leq y_{0}\leq u_0$. Assume that $\inf\limits_{s\leq t}(u_s-l_s)>0$ for every $t\geq 0$. Let $\big(x, k=(\phi^1,\phi^2)\big)$ and $(\tilde{x},\tilde{k}=(\tilde{\phi}^1,\tilde{\phi}^2)\big)$  be two solutions associated to $RP^u_l(y)$. Then, we have:
\begin{itemize}
\item[(i)] \begin{equation*}
\{t : \Delta^{+}\phi^1_t>0\}\cap\{t : \Delta^{+}\phi^2_t>0\}=\emptyset.\label{9}
\end{equation*}
\item[(ii)] \begin{equation*}
\{t : \Delta^{+}\phi^1_t>0\} =\{t : \Delta^{+}\tilde{\phi}^1_t>0\}; \label{10}
\end{equation*}
\begin{equation*}
\{t : \Delta^{+}\phi^2_t>0\} =\{t : \Delta^{+}\tilde{\phi}^2_t>0\}. \label{11}
\end{equation*}
\end{itemize}
\end{lemma}
\begin{proof}	
	\begin{itemize}
\item[(i)] Assume that there exists $t$ such that $\Delta^{+}\phi^1_t>0$ and $\Delta^{+}\phi^2_t>0$, from the statement (v) in Definition \ref{def1}, we get $x_t=u_t=l_t$ and $x_{t^+}=l_{t^+}=u_{t^+}$. This contradicts the fact that 
 $\inf\limits_{s\leq t}(u_s-l_s)>0$ for every $t\geq 0$, then $\{s : \Delta^{+}\phi^1_s>0\}\cap\{s : \Delta^{+}\phi^2_s>0\}=\emptyset$.
 \item[(ii)] Let $t\in\{s : \Delta^{+}\phi^1_s>0\}$, then $x_t=u_t$ and $x_{t^+}=l_{t^+}$, and from the first statement of this Lemma, we get
$$\Delta^+y_t=l_{t^+}-u_t-\Delta^{+}\phi^1_t.$$
On the other hand, we have
\begin{equation*}
\begin{split}
\Delta^+\tilde{\phi}^1_t&=\Delta^+\tilde{x}_t-\Delta^+y_t\\
&=\Delta^+\tilde{x}_t-l_{t^+}+u_t+\Delta^{+}\phi^1_t.
\end{split}
\end{equation*}
 From the statement (ii) in Definition \ref{def1}, $\Delta^+\tilde{x}_t\geq l_{t^+}-u_t$, then we obtain
 $$\Delta^+\tilde{\phi}^1_t\geq \Delta^{+}\phi^1_t>0,$$
this implies,
$$ \{t : \Delta^{+}\phi^1_t>0\} \subseteq\{t : \Delta^{+}\tilde{\phi}^1_t>0\}. $$
 Similarly, we show that  
 $$\{t : \Delta^{+}\tilde{\phi}^1_t>0\}\subseteq \{t : \Delta^{+}\phi^1_t>0\}
 \;\text{ and }\;\{t : \Delta^{+}\phi^2_t>0\} =\{t : \Delta^{+}\tilde{\phi}^2_t<0\}.$$
 	\end{itemize}
\end{proof}	
The following result establishes the uniqueness of the solution of the reflection problem RP().
\begin{theorem}\label{uniqueness}
Let  $y$, $l$ and $u$ be in  $\mathcal{R}(\mathbb{R}^{+},\mathbb{R})$  such that $l\leq u$ and  $l_0\leq y_{0}\leq u_0$. Suppose that $\inf\limits_{s\leq t}(u_s-l_s)>0$ for every $t\geq 0$, then, the uniqueness of  reflection problem $RP^u_l(y)$ holds.
\end{theorem}
\begin{proof}
Let $\big(x, k=(\phi^1,\phi^2)\big)$ and $(\tilde{x},\tilde{k}=(\tilde{\phi}^1,\tilde{\phi}^2)\big)$ be two solutions of $RP^u_l(y)$.  We set $A= x-\tilde{x}=
k-\tilde{k}$, which has a bounded variation. It follows from Proposition \ref{r3} that:
\begin{equation*}
\begin{split}
(x_t-\tilde{x}_t)^2&\leq  \int_{]0,t]} (x_s-\tilde{x}_s+x_{s^+}-\tilde{x}_{s^+}) \rm{d}(\phi^{1,r}-\tilde{\phi}^{1,r})_s -\int_{]0,t]} (x_s-\tilde{x}_s+x_{s^+}-\tilde{x}_{s^+}) \rm{d}(\phi^{2,r}-\tilde{\phi}^{2,r})_s\\
&+2\sum_{0\leq s<t}(x_{s^+}-\tilde{x}_{s^+})(\Delta^+\phi^1_s-\Delta^+\tilde{\phi}^1_s)-2\sum_{0\leq s<t}(x_{s^+}-\tilde{x}_{s^+})(\Delta^+\phi^2_s-\Delta^+\tilde{\phi}^2_s)\\
&-\sum_{0<s\leq t}(\Delta^+x_s-\Delta^+\tilde{x}_s)(\Delta^-x_s-\Delta^-\tilde{x}_s),\\
\end{split}
\end{equation*}
Using the  formula: $p\land q=\frac{1}{2}(p+q-|p-q|)$, we obtain:
\begin{equation*}
\begin{split}
x_{s}-\tilde{x}_{s}+ x_{s^{+}}-\tilde{x}_{s^{+}} &=2 \big((x_{s}-l_{s})\land (x_{s^{+}}-l_{s^{+}})\big)-2\big(( \tilde{x}_{s}-l_{s})\land (\tilde{x}_{s^{+}}-l_{s^{+}})\big) \\
&+|\Delta^{+} x_{s}-\Delta^{+} l_{s}|-|\Delta^{+} \tilde{x}_{s}-\Delta^{+} l_{s}|\\
&=-2 \Big(\big((u_{s}-x_{s})\land (u_{s^{+}}-x_{s^{+}})\big)-\big(( u_{s}-\tilde{x}_{s})\land (u_{s^{+}}-\tilde{x}_{s^{+}})\big)\Big)\\
&-|\Delta^{+} x_{s}-\Delta^{+} u_{s}|+|\Delta^{+} \tilde{x}_{s}-\Delta^{+} u_{s}|.
\end{split}
\end{equation*}
Consequently, 
\begin{equation*}
\begin{split}
(x_{t}-\tilde{x}_{t})^{2}&\leq 2\int_{]0,t]}\big( (x_{s}-l_{s})\land (x_{s^{+}}-l_{s^{+}})\big)-\big(( \tilde{x}_{s}-l_{s})\land (\tilde{x}_{s^{+}}-l_{s^{+}}) \big)\rm{d}(\phi^{1,r}_{s}-\tilde{\phi}_{s}^{1,r})\\
&+ 2\int_{]0,t]}  \big((u_{s}-x_{s})\land( u_{s^{+}}-x_{s^{+}})\big)-\big((u_{s}-\tilde{x}_{s})\land( u_{s^{+}}-\tilde{x}_{s^{+}})\big) \rm{d}(\phi^{2,r}_{s}-\tilde{\phi}_{s}^{2,r})\\
&+ \int_{]0,t]} \big(|\Delta^{+} x_{s}-\Delta^{+} l_{s}|-|\Delta^{+} \tilde{x}_{s}-\Delta^{+} l_{s}|\big) \rm{d}(\phi^{1,r}_{s}-\tilde{\phi}_{s}^{1,r})\\
& +\int_{]0,t]} \big(|\Delta^{+} x_{s}-\Delta^{+} u_{s}|-|\Delta^{+} \tilde{x}_{s}-\Delta^{+} u_{s}|\big) \rm{d}(\phi_{s}^{2,r}-\tilde{\phi}_{s}^{2,r})\\
&+2\sum_{0\leq s<t}(x_{s^+}-\tilde{x}_{s^+})(\Delta^+\phi^1_s-\Delta^+\tilde{\phi}^1_s)-2\sum_{0\leq s<t}(x_{s^+}-\tilde{x}_{s^+})(\Delta^+\phi^2_s-\Delta^+\tilde{\phi}^2_s)\\
&-\sum_{0<s\leq t}(\Delta^+x_s-\Delta^+\tilde{x}_s)(\Delta^-x_s-\Delta^-\tilde{x}_s)\\
&= I_{1}(t)+I_{2}(t)+I_{3}(t)+I_{4}(t)+I_{5}(t)+I_{6}(t)+I_{7}(t).
\end{split}
\end{equation*}
On one hand, it follows from the statements (ii), (iv) and (v) in Definition  \ref{def1} that $I_{1}$,  $I_{2}$, $I_5$ and $I_6$  are non-positive. On the other hand, for every $s\in [0,t]$ we have
\begin{equation*}
\begin{split}
\big(|\Delta^{+} x_{s}-\Delta^{+} l_{s}|-|\Delta^{+} \tilde{x}_{s}-\Delta^{+} l_{s}|\big)&= \big(|\Delta^{+} x_{s}-\Delta^{+} l_{s}|-|\Delta^{+} \tilde{x}_{s}-\Delta^{+} l_{s}|\big)1_{\{\Delta^+k_s>0\}}\\
&+ \big(|\Delta^{+} x_{s}-\Delta^{+} l_{s}|-|\Delta^{+} \tilde{x}_{s}-\Delta^{+} l_{s}|\big)1_{\{\Delta^+k_s<0\}} \\
&+ \big(|\Delta^{+} x_{s}-\Delta^{+} l_{s}|-|\Delta^{+} \tilde{x}_{s}-\Delta^{+} l_{s}|\big)1_{\{\Delta^+k_s=0\}} 
\end{split}
\end{equation*}
from the statement (v) in Definition \ref{def1} and Lemma \ref{lem1}, we obtain:
\begin{equation*}
\begin{split}
\big(|\Delta^{+} x_{s}-\Delta^{+} l_{s}|-|\Delta^{+} \tilde{x}_{s}-\Delta^{+} l_{s}|\big)& =
\big(| l_{s}-u_{s}|-| l_{s}-u_{s}|\big)1_{\{\Delta^+k_s>0\}} \\
&+ \big(|u_{s^+}-l_{s^+}|-|u_{s^+}-l_{s^+}|\big)1_{\{\Delta^+k_s<0\}} \\
&+ \big(|\Delta^{+} y_{s}-\Delta^{+} l_{s}|-|\Delta^{+} y_{s}-\Delta^{+} l_{s}|\big)1_{\{\Delta^+k_s=0\}} )\\
&=0.
\end{split}
\end{equation*}
Similarly, using again the statement (v) in Definition \ref{def1} and Lemma \ref{lem1}, we get
$$ \big(|\Delta^{+} x_{s}-\Delta^{+} u_{s}|-|\Delta^{+} \tilde{x}_{s}-\Delta^{+} u_{s}|\big)=0,\;\text{ and }\;(\Delta^+x_s-\Delta^+\tilde{x}_s)(\Delta^-x_s-\Delta^-\tilde{x}_s)=0.$$
Therefore, $I_3=I_4=I_7=0$. 
This implies $x=\tilde{x}$ and $k=\tilde{k}$. This completes the proof.
\end{proof}

We introduce the following mappings
$\Theta,\beta:\mathcal{R}(\mathbb{R}^{+},\mathbb{R})\times\mathcal{R}(\mathbb{R}^{+},\mathbb{R})\times\mathcal{R}(\mathbb{R}^{+},\mathbb{R})\to\mathcal{R}(\mathbb{R}^{+},\mathbb{R})$ and
$\alpha:\mathcal{R}(\mathbb{R}^{+},\mathbb{R})\times\mathcal{R}(\mathbb{R}^{+},\mathbb{R})\to\mathcal{R}(\mathbb{R}^{+},\mathbb{R})$ defined by:
\begin{equation*}
\Theta(y,l,u)_t:=\Theta^{u}_{l}(y)_t=\sup_{s\leq t}\Bigg( \big((y_s-u_s)\lor (y_{s^{+}}-u_{s^+})\big)^+\land \inf_{s\leq r\leq t}\Big[\Big((y_r-l_r)\land (y_{r^{+}}-l_{r^+})\Big)\lor(y_r-u_r)\Big]\Bigg),
\end{equation*}
\begin{equation*}
\beta(y,l,u)_t:=\beta^{u}_{l}(y)_t=\sup_{s\leq t}\Bigg( \big((y_s-u_s)\lor (y_{s^{+}}-u_{s^+})\big)\land \inf_{s\leq r\leq t}\Big[\Big((y_r-l_r)\land (y_{r^{+}}-l_{r^+})\Big)\lor(y_r-u_r)\Big]\Bigg),\label{13}
\end{equation*}
and
\begin{equation*}
\alpha(y,l)_t:=\alpha^{l}(y)_t=\inf_{s\leq t}\big((y_s-l_s)\land (y_{s^{+}}-l_{s^{+}})\land 0\big).\label{12}
\end{equation*}	
The following Theorem  ensures the existence of the reflection problem $RP()$.
\begin{theorem}\label{th1}
 Let  $y$,$l$ and $u$ be in  $\mathcal{R}(\mathbb{R}^{+},\mathbb{R})$  such that $l\leq u$ and  $l_0\leq y_{0}\leq u_0$. Assume that $\inf\limits_{s\leq t}(u_s-l_s)>0$ for every $t\geq 0$. Then, the reflection problem $RP^u_l(y)$ has an explicit solution, which is given by:
$$k_t=-\Big(\alpha^{l}(y)_t\lor \beta^{u}_l(y)_t\Big),\; x_t=y_t+k_t,\;\;\text{for every } t\geq 0.$$

\end{theorem}
The proof of Theorem \ref{th1} is divided into a sequence of Lemmas.
\begin{lemma}\label{lemma1}
 Let  $y$, $l$, and $u$ be in  $\mathcal{R}(\mathbb{R}^{+},\mathbb{R})$  such that $l\leq u$, $l_0\leq y_{0}\leq u_0$, and $\inf\limits_{s\leq t}(u_s-l_s)>0$ for every $t\geq 0$. Let $(\xi,\kappa)=RP_l(y)$, then
$\Theta^u_l(\xi)$  has a bounded variation. 

Here $RP_l(y)$ denotes the reflection problem with one lower barrier (see Proposition \ref{r1}).
\end{lemma}
\begin{proof}

We consider the following sequence of times:
$$s_0=0,\;\;t_0=\inf\{t\geq 0:  (\xi_t-u_t) \geq 0 \}$$ 
for $n\geq 1$,
$$ s_n= \inf\{t\geq t_{n-1}:  (\xi_t-l_t)\leq \sup_{t_{n-1}\leq s\leq t}(\xi_s-u_s) \},$$
$$t_n= \inf\{t\geq s_{n}:  \inf_{s_{n}\leq s\leq t}(\xi_s-l_s)\leq (\xi_t-u_t) \},$$

with $\inf(\emptyset)=+\infty.$ For $s\geq 0$, we set:
 $$g_s=\big((\xi_{s}-u_s)\lor(\xi_{s^{+}}-u_{s^+})\big)^+\;\text{ and }\; h_s=\Big(\big((\xi_s-l_s)\land(\xi_{s^{+}}-l_{s^{+}})\big)\lor(\xi_s-u_s)\Big).$$
The proof is divided into 3 steps
\begin{itemize}
\item[\textbf{Step 1:}]  We show that  the following  statements hold: 
\begin{itemize}
\item[(i)] $\text{For all } \;\;t\in[s_0,t_0[,\;\;\Theta^u_l(\xi)_t=0 $

\item[(ii)]$\text{For every}\;\;t\in [t_{n-1},s_n[,$ \begin{equation*}
 \;\;\;\Theta^u_l(\xi)_t=\sup_{t_{n-1}\leq s\leq t}(g_s\land(\xi_s-l_s)).
\end{equation*}
\item[(iii)]$\text{For all } t\in[s_n,t_{n}[$ \begin{equation*}
\;\;\Theta^u_l(\xi)_t= \inf_{s_n\leq s\leq t }h_s. 
\end{equation*}
\end{itemize}
\begin{itemize}
\item[(i)] Let us show that $\text{for all} \;t\in[s_0,t_0[,\;\;\Theta^u_l(\xi)_t=0$.
Let $s\in[0,t_0[$.
Observe that $\big((\xi_{s}-u_s)\lor(\xi_{s^{+}}-u_{s^{+}})\big)\leq 0 $,
then $g_s=0$. Since $ (\xi,\kappa)=RP_l(y)$, then $h_s\geq 0$. Therefore the statement  (i) holds.
\item[(ii)] Let $n\geq 1$, and let $t\in [t_{n-1},s_n[$, we have:
\begin{equation}
\begin{split}
\Theta^u_l(\xi)_t=\max\Big(\sup_{s< t_{n-1}}(g_s\land \inf_{s\leq r\leq t}h_r), \sup_{t_{n-1}\leq s\leq t}(g_s\land \inf_{s\leq r\leq t}h_r)\Big).\label{19}
\end{split}
\end{equation} 
Let $s\in [t_{n-1},t]$ and  $r\in]s,t]$, then
\begin{equation}
(\xi_s-u_s)\leq \sup_{t_{n-1}\leq v\leq r}(\xi_v-u_v)\leq \xi_r-l_r, \label{14}
\end{equation}
 since $r<s_n$, there exists $i_{n}$ such that for every $i\geq i_{n}$, $r+\frac{1}{i}<s_n$, then for every $i\geq i_{n}$
$$ (\xi_s-u_s)\leq \sup_{t_{n-1}\leq v\leq r+\frac{1}{i}}(\xi_v-u_v)\leq \xi_{r+\frac{1}{i}}-l_{r+\frac{1}{i}}, $$
tending $i$ to infinity, we obtain
\begin{equation}
(\xi_s-u_s)\leq (\xi_{r^{+}}-l_{r^{+}}), \label{15}
\end{equation}
similarly, we obtain 
\begin{equation}
(\xi_{s^+}-u_{s^+})\leq (\xi_{r^{+}}-l_{r^{+}}). \label{16}
\end{equation}
Since $r>s$,
\begin{equation}
(\xi_{s^{+}}-u_{s^{+}})\leq\sup_{t_{n-1}\leq v<r}(\xi_{v^{+}}-u_{v^{+}})\leq  \sup_{t_{n-1}\leq v\leq r}(\xi_v-u_v)\leq (\xi_r-l_r). \label{17}
\end{equation}

From \eqref{14},\eqref{15},\eqref{16}, and \eqref{17}, we obtain $ g_s\leq h_r,$ therefore, 
\begin{equation}
g_s\land h_s\leq \inf_{s<r\leq t}h_r, \label{118}
\end{equation}
and then, we get
\begin{equation}
 \sup_{t_{n-1}\leq s\leq t}(g_s\land \inf_{s\leq r\leq t}h_r)=\sup_{t_{n-1}\leq s\leq t}(g_s\land h_s)= \sup_{t_{n-1}\leq s\leq t}(g_s\land(\xi_s-l_s)), \label{18}
\end{equation}
where in the last equality we have used the fact that $l_+\leq u_+$.

On the other hand, we have
\begin{equation}
\begin{split}
\sup_{s< t_{n-1}}(g_s\land \inf_{s\leq r\leq t}h_r)&=\max\Big(\sup_{s< s_{n-1}}(g_s\land \inf_{s\leq r\leq t}h_r), \sup_{s_{n-1}\leq s< t_{n-1}}(g_s\land \inf_{s\leq r\leq t}h_r)\Big).
\label{1}
\end{split}
\end{equation}
Let $s\in [s_{n-1},t_{n-1}[$, we have 
$$\big((\xi_s-u_s)\lor(\xi_{s^{+}}-u_{s^{+}})\big)\leq \inf_{s_{n-1}\leq r\leq s}(\xi_r-l_r)\leq \inf_{s_{n-1}\leq r<s}(\xi_r-l_r)=\inf_{s_{n-1}\leq r<s}\big((\xi_r-l_r)\land(\xi_{r^{+}}-l_{r^{+}})\big), $$
then for every $s \in[s_{n-1},t_{n-1}[$,
\begin{equation}
 g_s\leq \inf_{s_{n-1}\leq r<s}h_r .\label{2}
\end{equation}
Combining  \eqref{1} with \eqref{2}, we get
\begin{align}
\sup_{s< t_{n-1}}(g_s\land \inf_{s\leq r\leq t}h_r)&\leq\max\Big(\inf_{s_{n-1}\leq r\leq t_{n-1}}(h_r), \sup_{s_{n-1}\leq s< t_{n-1}}\big(\inf_{s_{n-1}\leq r<s}(h_r)\land \inf_{s\leq r\leq t}(h_r)\big)\Big)\nonumber\\
&= \inf_{s_{n-1}\leq r\leq t_{n-1}}h_r.\label{20}
\end{align}
From the expression of $h$, we have
$$\inf_{s_{n-1}\leq r\leq t_{n-1}}h_r\leq \inf_{s_{n-1}\leq r\leq t_{n-1}}(\xi_r-l_r), $$
and 
$$\inf_{s_{n-1}\leq r\leq t_{n-1}}h_r\leq (\xi_{t_{n-1}^+}-l_{t_{n-1}^+})\lor(\xi_{t_{n-1}}-u_{t_{n-1}})  $$
then
\begin{equation}
\inf_{s_{n-1}\leq r\leq t_{n-1}}h_r\leq \big(\inf_{s_{n-1}\leq r\leq t_{n-1}}(\xi_r-l_r)\big)\land\big( (\xi_{t_{n-1}^+}-l_{t_{n-1}^+})\lor(\xi_{t_{n-1}}-u_{t_{n-1}})\big).\label{33}
\end{equation}
By definition of $t_{n-1}$, we have
$$(\inf_{s_{n-1}\leq r\leq t_{n-1}}(\xi_r-l_r))\land(\xi_{t_{n-1}^{+}}-l_{t_{n-1}^{+}})\leq \big((\xi_{t_{n-1}}-u_{t_{n-1}})\lor(\xi_{t_{n-1}^{+}}-u_{t_{n-1}^{+}})\big),$$
then
$$\big(\inf_{s_{n-1}\leq r\leq t_{n-1}}(\xi_r-l_r)\big)\land\big( (\xi_{t_{n-1}^+}-l_{t_{n-1}^+})\lor(\xi_{t_{n-1}}-u_{t_{n-1}})\big)\leq \big((\xi_{t_{n-1}}-u_{t_{n-1}})\lor(\xi_{t_{n-1}^{+}}-u_{t_{n-1}^{+}})\big),$$
from the previous inequality and \eqref{33}, we get
$$\inf_{s_{n-1}\leq r\leq t_{n-1}}h_r\leq (\xi_{t_{n-1}}-u_{t_{n-1}})\lor(\xi_{t_{n-1}^{+}}-u_{t_{n-1}^{+}}),   $$
 since $\inf_{s_{n-1}\leq r\leq t_{n-1}}h_r\leq (\xi_{t_{n-1}}-l_{t_{n-1}}), $
we obtain 
$$\inf_{s_{n-1}\leq r\leq t_{n-1}}h_r\leq g_{t_{n-1}}\land(\xi_{t_{n-1}}-l_{t_{n-1}})\leq \sup_{t_{n-1}\leq s\leq t}(g_s\land(\xi_s-l_s)) .$$
From \eqref{19},\eqref{18}, \eqref{20} and the previous inequality,  we get (ii).
\item[(iii)]  Let $t\in[s_n,t_{n}[$, we have 
 \begin{equation*}
\begin{split}
\Theta^u_l(\xi)_t&=\max(\sup_{s\leq s_n}(g_s\land \inf_{s\leq r\leq t}h_r),\sup_{s_n< s\leq t}(g_s\land \inf_{s\leq r\leq t}h_r))\\
&=\max(\sup_{s\leq s_n}(g_s\land \inf_{s\leq r\leq s_n}h_r)\land\inf_{s_n\leq r\leq t}h_r,\sup_{s_n< s\leq t}(g_s\land \inf_{s\leq r\leq t}h_r)).\label{35}
\end{split}
 \end{equation*}
On one hand, we have
 \begin{equation*}
 \begin{split}
 \sup_{s\leq s_n}(g_s\land \inf_{s\leq r\leq s_n}h_r)&\geq \sup_{t_{n-1}\leq s\leq s_n}(g_s\land \inf_{s\leq r\leq  s_n}h_r)\\
 \end{split}
 \end{equation*}
 and 
 \begin{equation*}
\begin{split}
 \sup_{t_{n-1}\leq s\leq s_n}(g_s\land \inf_{s\leq r\leq  s_n}h_r)&=\max(\sup_{t_{n-1}\leq s< s_n}(g_s\land \inf_{s\leq r<  s_n}h_r)\land h_{s_n},g_{s_n}\land h_{s_n})\\
 &=\max(\sup_{t_{n-1}\leq s< s_n}(g_s\land h_s)\land h_{s_n},g_{s_n}\land h_{s_n})\\
&=   \sup_{t_{n-1}\leq s\leq s_n}(g_s\land h_s)
\end{split}
 \end{equation*}
 where in the second equality, we have used \eqref{118}. Since $l\leq u$, we get
$$\sup_{t_{n-1}\leq s\leq s_n}(g_s\land h_s)\geq \sup_{t_{n-1}\leq s\leq s_n}(\xi_s-u_s), $$

and by the definition of $s_n$, we have 
$$ \sup_{t_{n-1}\leq s\leq s_n}(\xi_s-u_s)\geq h_{s_n}\geq \inf_{s_{n}\leq r\leq t}h_r, $$

therefore
\begin{equation}
\Theta^u_l(\xi)_t=\max(\inf_{s_n\leq r\leq t}h_r,\sup_{s_n< s\leq t}(g_s\land \inf_{s\leq r\leq t}h_r)).\label{21}
\end{equation}
On the other hand, it follows from  \eqref{2} that
\begin{equation*}
\begin{split}
\sup_{s_n< s\leq t}(g_s\land \inf_{s\leq r\leq t}h_r)&\leq\sup_{s_n< s\leq t}(\inf_{s_n\leq r< s}h_r\land \inf_{s\leq r\leq t}h_r) = \inf_{s_n\leq r\leq t }h_r.
\end{split}
\end{equation*}
From \eqref{21} and the previous inequality, we get (iii).
\end{itemize}
 \item[ \textbf{Step 2 }] Let $n\geq 1$. We show  that $s_n< s_{n+1}$ if $s_{n+1}<\infty$  (respectively $t_{n}< t_{n+1}$ if  $t_{n+1}<\infty$) . 
 
 Let $n\geq 1$ such that $s_{n+1}<\infty$. Assume that $s_{n+1}=s_{n}$. Since $s_n\leq t_{n}\leq s_{n+1}$, $t_{n}=s_n.$
 
  By the definition of $t_{n}$ and since $s_n=t_{n}$, we obtain
$$\xi_{s_{n}}-l_{s_{n}}\leq \xi_{s_{n}}-u_{s_{n}}\;
\text{ or }\;
(\xi_{s_{n}}-l_{s_{n}})\land(\xi_{s_{n}^{+}}-l_{s_{n}^{+}})\leq \xi_{s_{n}^{+}}-u_{s_{n}^{+}},$$
since $l<u $ and $l_+<u_+$, we get
\begin{equation}
\xi_{s_{n}}-l_{s_{n}}\leq \xi_{s_{n}^{+}}-u_{s_{n}^{+}}. \label{3}
\end{equation}
On the other hand, by the definition of $s_{n+1}$ and the fact that $s_n=s_{n+1}=t_{n}$, we get
$$\xi_{s_{n}}-l_{s_{n}}\leq \xi_{s_{n}}-u_{s_{n}}\;\text{ or }\;\xi_{s_{n}^{+}}-l_{s_{n}^{+}}\leq (\xi_{s_{n}}-u_{s_{n}})\lor(\xi_{s_{n}^{+}}-u_{s_{n}^{+}}) $$
again, from the fact that $l<u$ and $l_+<u_+$,  we obtain
\begin{equation}
\xi_{s_{n}^{+}}-l_{s_{n}^{+}}\leq (\xi_{s_{n}}-u_{s_{n}}). \label{4}
\end{equation}
From \eqref{3} and \eqref{4}, we get
$$u_{s_n^+}-l_{s_n}\leq \Delta^{+}\xi_{s_n}\leq l_{s_n^+} -u_{s_n}$$
which implies that
$$0\leq u_{s_n^+}-l_{s_n^+}\leq l_{s_n}-u_{s_n}\leq 0 . $$
thus,
$$ u_{s_n^+}=l_{s_n^+}\;\;\text{and}\;\; l_{s_n}=u_{s_n}.$$
This contradicts the assumption $\inf\limits_{s\leq t}(u_s-l_s)>0$ for every $t\geq 0$.
Therefore, $s_n<s_{n+1}.$
Similarly, we show that $t_n<t_{n+1}$ if $t_{n+1}<\infty$.
\item[\textbf{Step 3}] We show that $\lim\limits_{n\to+\infty} s_n=+\infty.$

Assume that $\lim\limits s_{n}=s<\infty$. Let $n\geq 1$, we have  by the definition of $s_n$

$$(\xi_{s_n}-l_{s_n})\land (\xi_{s_{n}^+}-l_{s_{n}^+})\leq \sup_{t_{n-1}\leq s\leq s_{n}}(\xi_{s}-u_{s})\lor (\xi_{s_{n}^+}-u_{s_{n}^+}) $$
 tending $n$ to infinity in the inequality above, it follows  from  step 2 that 
$$\xi_{l^-}-l_{l^-}\leq \xi_{l^-}-u_{l^-}, $$
which contradict the fact that $l_-<u_-$.
  \end{itemize}
 From the three previous steps, we conclude that
 \begin{multline}
 \Theta^u_l(\xi)_t=\sum_{i\geq 0}\sup_{t_{i}\leq s\leq t}\Big(\big((\xi_{s}-u_s)\lor(\xi_{s^{+}}-u_{s^{+}})\big)^+\land(\xi_s-l_s)\Big) 1_{[t_{i},s_{i+1}[}(t)\\+\sum_{i\geq 1}\inf_{s_i\leq s\leq t }\Big(\big((\xi_s-l_s)\land(\xi_{s^{+}}-l_{s^{+}})\big)\lor(\xi_s-u_s)\Big) 1_{[s_{i},t_i[}(t). \label{8}
 \end{multline}
 This shows that $\Theta^u_l(\xi)$ has a bounded variation.
\end{proof}
\begin{lemma}\label{prop3}
 Let  $y$, $l$ and $u$ be in  $\mathcal{R}(\mathbb{R}^{+},\mathbb{R})$  such that $l\leq u$ and $l_0\leq y_{0}\leq u_0$. Let $(\xi,\kappa)$ be the solution of $RP_l(y)$, then for all $t\geq 0$ we have,
 \begin{equation*}
\big(\alpha^{l}(y)_t\lor\beta^{u}_l(y)_t\big)=\alpha^{l}(y)_t + \Theta^u_l(\xi)_t. \label{29}
 \end{equation*}
\end{lemma}

\begin{proof}
Let $t\geq 0$, we have
\begin{equation}
\begin{split}
\Big(\alpha^{l}(y)_t\lor\beta^{u}_l(y)_t\Big)=\alpha^l(y)_t+\big(\beta^{u}_l(y)_t-\alpha^l(y)_t\big)^+.\label{5}\\
\end{split}
\end{equation}
We set $\gamma:=\beta^{u}_l(y)-\alpha^l(y)$. For simplicity, we write $\beta$ and $\alpha$ instead of $\beta^{u}_l(y)$ and $\alpha^l(y)$, respectively.  We show that $\gamma^+=\Theta^u_l(\xi)$. We have
\begin{multline*}
\gamma_t=\sup_{s\leq t}\Bigg( \big((y_s-\alpha_t-u_s)\lor (y_{s^{+}}-\alpha_t-u_{s^{+}})\big)\land \inf_{s\leq r\leq t}\Big(\big((y_r-\alpha_t-l_r)\land (y_{r^{+}}-\alpha_t-l_{r^{+}})\big)\lor(y_r-\alpha_t-u_r)\Big)\Bigg).\\
\end{multline*}
From Proposition \ref{r1}, we have
 $\xi=y-\alpha$, then $\gamma$ can be rewritten as
\begin{equation*}
\gamma_t=\sup_{s\leq t}( \gamma^1_s\land \inf_{s\leq r\leq t}\gamma^2_r),\label{gam}
\end{equation*}
where $$\gamma^1_s= \big((\xi_s+\alpha_s-\alpha_t-u_s)\lor (\xi_{s^+}+\alpha_s-\alpha_t-u_{s^+})\big)$$ and $$\gamma^2_s=\Big(\big((\xi_s+\alpha_s-\alpha_t-l_s)\land (\xi_{s^+}+\alpha_s-\alpha_t-l_{s^+})\big)\lor(\xi_s+\alpha_s-\alpha_t-u_s)\Big).$$
Let $s\leq t$ and $r\in[s,t]$. If for every $r\in [s,t]$, $(\xi_r-l_r)\land(\xi_{r^+}-l_{r^+})>0$, then for every $r\in[s,t]$  $\alpha_s
=\alpha_t=\alpha_r$ . Therefore,
\begin{equation*}
\gamma_t=\sup_{s\leq t}\Bigg( \big((\xi_s-u_s)\lor (\xi_{s^{+}}-u_{s^{+}})\big)\land \inf_{s\leq r\leq t}\Big(\big((\xi_r-l_r)\land (\xi_{r^{+}}-l_{r^{+}})\big)\lor(\xi_r-u_r)\Big)\Bigg).
\end{equation*}
 Since $\xi\geq l$, then
 \begin{equation*}
\gamma_t^+= \sup_{s\leq t}\Big( \big((\xi_s-u_s)\lor (\xi_{s^{+}}-u_{s^{+}})\big)^+\land \inf_{s\leq r\leq t}\Big(\big((\xi_r-l_r)\land (\xi_{r^{+}}-l_{r^{+}})\big)\lor(\xi_r-u_r)\Big)\Bigg)=\Theta^u_l(\xi)_t.
 \end{equation*}
 If  there exists  $r$ in $[s,t]$ such that $(\xi_r-l_r)\land(\xi_{r^+}-l_{r^+})=0$.
 We consider $\tau_r=\sup\{v\in[s,t],(\xi_v-l_v)\land(\xi_{v^{+}}-l_{v^{+}})=0 \}$. We distinguish two cases:
 \begin{itemize}
\item[ 	\textbf{Case 1}:] $\tau_r=t $, $
 \xi_{\tau_r^+}=l_{\tau_r^+}$ and $\xi_{\tau_r}>l_{\tau_r}$.
 By the definition of $\gamma$ and $\Theta^u_l(\xi)$, we have
 \begin{equation*}
 \begin{split}
\Theta^u_l(\xi)_t\lor\gamma_t\leq \Big(\big((\xi_t-l_t)\land (\xi_{t^+}-l_{t^+})\big)\lor(\xi_t-u_t)\Big),
 \end{split}
 \end{equation*}
Since $\xi_{t^+}=l_{t^+}$ and $\xi_{t}>l_{t}$, we obtain
 \begin{equation*}
\Theta^u_l(\xi)_t\lor\gamma_t^+\leq (\xi_t-u_t)^+ \label{6}
 \end{equation*}
From the expression  of $\gamma$ and $\Theta^u_l(\xi)$, we have $\gamma_t^+\geq (\xi_t-u_t)^+$ and $\Theta^u_l(\xi)\geq (\xi_t-u_t)^+$, combining this with the last inequality, we obtain $$\gamma_t^+=\Theta^u_l(\xi)= (\xi_t-u_t)^+.$$
\item[\textbf{Case 2:}]  $\tau_r<t $ or $
\xi_{\tau_r^+}>l_{\tau_r^+}$ or $\xi_{\tau_r}=l_{\tau_r}$.
 
 From proposition \ref{r1}, $\alpha$ is right continuous, then we have two cases to distinguish, either $(\xi_{\tau_r}-l_{\tau_r})\land(\xi_{\tau_r^+}-l_{\tau_r^+})=0$ and $\alpha $ is constant on $[\tau_r,t]$, or $\xi_{\tau_r^-}=l_{\tau_r^-}$, $\alpha_{\tau_r}=\alpha_{\tau_r^-}$ and $\alpha$ is constant on $[\tau_r,t]$.
 It follows, in  both cases, that $$\gamma_t^+=\Theta^u_l(\xi)_t=0.$$ 
  \end{itemize}
\end{proof}

\it{\textbf{Proof of Theorem \ref{th1}}.}
Set $k_t=-\Big(\alpha^{l}(y)_t\lor \beta^{u}_l(y)_t\Big)$ and $ x=y+k$, let us  show that $(x,k)$ is a solution of $RP_l^u(y)$.
	\begin{itemize}
\item[i)]Proof of statement (ii) of Definition \ref{def1}.

Let $t\geq 0$, on  one hand,  from the expression of $\alpha^l(y) $ and $\beta^u_l(y)$ we have, $$\alpha_t^l(y)\leq y_t-l_t,$$ and $$\beta^u_l(y)_t\leq \big((y_t-l_t)\land(y_{t^+}-l_{t^+})\big)\lor(y_{t}-u_t)\leq (y_t-l_t)\lor(y_{t}-u_t)=y_t-l_t,$$
therefore $$l_t\leq x_t= y_t+ k_t.$$

On the other hand,
$$\beta^u_l(y)_t\geq \big((y_t-u_t)\lor(y_{t^+}-u_{t^+})\big)\land\Big(\big((y_t-l_t)\land(y_{t^+}-u_{t^+})\big)\lor(y_{t}-u_t)\Big)\geq (y_t-u_t),$$
hence $$ y_t+k_t\leq u_t.$$
	
\item[ii)] Proof of statement (v) and (vi) of Definition \ref{def1}.

Let $(\xi,\kappa)=RP_l(y)$,  we use the same notations as in the proof of Lemma \ref{lemma1},we set
$$m_t^i=\inf_{s_i\leq s\leq t }\Big(\big((\xi_s-l_s)\land(\xi_{s^{+}}-l_{s^{+}})\big)\lor(\xi_s-u_s)\Big) , \;\text{for } i\geq 1$$
	and
$$n_t^i=\sup_{t_{i}\leq s\leq t}\Big(\big((\xi_{s}-u_s)\lor(\xi_{s^{+}}-u_{s^{+}})\big)^+\land(\xi_s-l_s)\Big) ,\; \;\text{for } i\geq 0.$$ 
It follows from Lemmas \ref{lemma1} and \ref{prop3} that $k$ has a bounded variation, then $k$ can be written as $k=k^r+\sum_{0\leq s<t}\Delta^+k_s$ and $k^r$ can be decomposed as $k^r=\phi^{1,r}-\phi^{2,r}$, where $\phi^{1,r}$ and $\phi^{2,r}$ are non-decreasing right continuous functions such that $\rm{d}\phi^{1,r}$ and $\rm{d}\phi^{2,r}$ have disjoint support. We set $\phi^{1,g}=\sum_{0\leq s<.}\Delta^+k_s1_{\Delta^+k_s>0}$ and $\phi^{2,g}=-\sum_{0\leq s<.}\Delta^+k_s1_{\Delta^+k_s<0}$, then $k:=\phi^1-\phi^2$ where $\phi^1=\phi^{1,r}+\phi^{1,g}$ and $\phi^2=\phi^{2,r}+\phi^{2,g}$

From \eqref{8}, $k$  also has the following  form:
$$k=-\alpha^l(y)-\sum_{i\geq 1}m^i1_{[s_{i},t_i[}-\sum_{i\geq 0}n^i1_{[t_{i},s_{i+1}[}.$$
Let $t\geq 0$ satisfying $\Delta^{+}k_t>0$, there exists $i$ such that $t\in [s_i,t_{i}[$ and $k_t=-\alpha^l(y)_t-m^i_t.$

Note that
\begin{equation*}
\begin{split}
m^i_{t^+}&=\lim\limits_{n\to+\infty}\inf_{s_i\leq s\leq t+\frac{1}{n} }\Big(\big((\xi_s-l_s)\land(\xi_{s^{+}}-l_{s^{+}})\big)\lor(\xi_s-u_s)\Big)\\
&=\lim\limits_{n\to+\infty} \min(m^i_{t}, \inf_{t< s\leq t+\frac{1}{n} }\Big(\big((\xi_s-l_s)\land(\xi_{s^{+}}-l_{s^{+}})\big)\lor(\xi_s-u_s)\Big) )\\
&=\min(m^i_{t},(\xi_{t^+}-l_{t^+})).
\end{split}
\end{equation*}
Since $k_{t^+}>k_t$ and $\alpha^l(y)$ is right continuous,
$m^i_{t^+}<m^i_{t}$,
hence
\begin{equation}
m^i_{t^+}= \xi_{t^{+}}-l_{t^{+}}, \label{40}
\end{equation}
then, we get 
\begin{equation*}
x_{t^+}=l_{t^+}.
\end{equation*}
It is easy to see that
\begin{equation*}
m^i_{t}=\min\Bigg(\inf_{s_i\leq s<t }\Big(\big((\xi_s-l_s)\land(\xi_{s^{+}}-l_{s^{+}})\big)\lor(\xi_s-u_s)\Big), \big((\xi_t-l_t)\land(\xi_{t^{+}}-l_{t^{+}})\big)\lor(\xi_t-u_t)\Bigg).
\end{equation*}
Since $m^i_{t^+}<m^i_t$, and from \eqref{40}, we get
 $$\big((\xi_t-l_t)\land(\xi_{t^{+}}-l_{t^{+}})\big)\lor(\xi_t-u_t)=\xi_t-u_t,$$
 it follows that
\begin{equation*}
m^i_{t}=\min\Bigg(\inf_{s_i\leq s<t }\Big(\big((\xi_s-l_s)\land(\xi_{s^{+}}-l_{s^{+}})\big)\lor(\xi_s-u_s)\Big), \xi_t-u_t\Bigg). \label{31}
\end{equation*}
Using the fact that $s_i\leq t<t_i$, we get
$$\xi_t-u_t\leq\inf_{s_i\leq s<t }\big(\xi_s-l_s)\leq \inf_{s_i\leq s<t }\Big(\big((\xi_s-l_s)\land(\xi_{s^{+}}-l_{s^{+}})\big)\lor(\xi_s-u_s)\Big).$$
Therefore, $m^i_{t}=\xi_t-u_t$,  hence $x_t=u_t$, therefore \eqref{d1} holds.\\
Now, we show that
\begin{equation*}
\int_{[0,+\infty[} \big((x_{s}-l_s)\land (x_{s^{+}}-l_{s^+})\big)\rm{d}\phi^{1,r}_{s}=0.
\end{equation*}
Let $t\geq 0$ such that $\Delta^+\phi^1_t=0$ and $\phi^1_{t^-}<\phi^1_s$, for every $s>t.$ We distinguish four cases:
\begin{itemize}
\item[$\bullet$]If $t\in]s_i,t_i[$ for some $i\geq 1$, then $k_t=-\alpha(y)_t-m^i_t$. Since $\rm{d}\phi^{1,r}$ and $\rm{d}\phi^{2,r}$ have disjoint support and $\phi^{2,g}_s=\phi^{2,g}_t$ for every $s\in]t,t_i[$, then $\Delta^+k_t=0$ and $k_{t^-}<k_s$ for every $s\in]t,t_i[$. We distinguish two cases:
\begin{itemize}
\item[\textbf{case 1:}]   $m^i_{t^-}>m^i_s$ for every $s\in]t,t_i[$,
 then for every $n\geq 1$, we have 
\begin{equation}
\begin{split}
m^i_{t+\frac{1}{n}}&=\min\Bigg(m^i_{t^-}, \inf_{t\leq s\leq t+\frac{1}{n} }\Big(\big((\xi_s-l_s)\land(\xi_{s^{+}}-l_{s^{+}})\big)\lor(\xi_s-u_s)\Big)\Bigg),\label{41}
\end{split}
\end{equation}
Using the fact that  $m^i_{t+\frac{1}{n}}< m^i_{t^-}$, we obtain:
\begin{equation*}
m^i_{t+\frac{1}{n}}=\inf_{t\leq s\leq t+\frac{1}{n} }\Big(\big((\xi_s-l_s)\land(\xi_{s^{+}}-l_{s^{+}})\big)\lor(\xi_s-u_s)\Big).
\end{equation*}
By tending $n$ to infinity, we obtain:
\begin{equation*}
m^i_{t^+}=\min\Big(\big((\xi_t-l_t)\land(\xi_{t^{+}}-l_{t^{+}})\big)\lor(\xi_t-u_t),\;\xi_{t^{+}}-l_{t^{+}} \Big)=(\xi_t-l_t)\land(\xi_{t^{+}}-l_{t^{+}}) \label{32}
\end{equation*}
Since $k_t=k_{t^+}$, we get
\begin{center}
 $x_t=l_t$ or $x_{t^+}=l_{t^+}$.
\end{center}
\item[\textbf{case 2:}]  If there exists $s'\in]t,t_i[$ such that $m^i_{t^-}=m^i_{s'}$, then  for every $s>t$, $\alpha^l(y)_{t^-}>\alpha^l(y)_{s}$ and
  $m^i_{t^-}=m^i_{t^+}$. By Proposition \ref{A1} $(\xi,-\alpha^l(y))=RP_l(y)$, then $\xi_t=l_t$ or $\xi_{t^+}=l_{t^+}$, and from \eqref{41} we get $m^i_{t^+}=\min(m^i_{t^-},0)$.
Since $m^i$ is non-negative,  $m^i_{t^+}=m^i_{t}=m^i_{t^-}=0$, therefore
$x=\xi$, hence $x_t=l_t$ or $x_{t^+}=l_{t^+}$.
\end{itemize}
\item[$\bullet$] If $t=s_i$ for some $i$, we have $\Delta^+k_t=0$, then
$$m^i_t=m^i_{t^+}=\min(m^i_t,\xi_{t^+}-l_{t^+} )=\min(\big((\xi_t-l_t)\land(\xi_{t^{+}}-l_{t^{+}})\big)\lor(\xi_t-u_t),\xi_{t^+}-l_{t^+} )$$
  this implies,  $m^i_t=\xi_t-l_t$,  hence $x_t=l_t$.

\item[$\bullet$] If $t=t_i$ for some $i$, then
$k_{t^-}=-\alpha^l(y)_{t^-}-m^i_{t^-}$ and $k_{t}=-\alpha^l(y)_{t}-n^i_{t}$. We distinguish two cases:
\begin{itemize}
\item[\textbf{case 1:}] If $\alpha^l(y)_{t^-}>\alpha^l(y)_{s}$ for every $s>t$, then $\xi_t=l_t$ or $\xi_{t^+}=l_{t^+}$, on the other hand, we have 
$x=\xi-n^i\leq \xi$, and since $x\geq l$, $x_t=l_t$ or $x_{t^+}=l_{t^+}$.
\item[\textbf{case 2:}] If there exists $s'>t$ such that $\alpha^l(y)_{t^-}=\alpha^l(y)_{s'}$. Using the fact that the right jumps of $k$ are non-positive on $[t_i,s_{i+1}[$, $\phi^1_{t^-}<\phi^1_s$, for $s>t$, and $\rm{d}\phi^{1,r}$ and $\rm{d}\phi^{2,r}$ have disjoint support,   we get $k_{t^-}<k_{t}$, then
$$m^i_{t^-}>n^i_{t}= \big((\xi_t-u_t)\lor(\xi_{t^{+}}-u_{t^{+}})\big)^+\land(\xi_t-l_t) $$
On one hand, by the definition of $t_i$, we have
\begin{equation*}
\begin{split}
\inf_{s_i\leq s\leq t}(\xi_s-l_s)\leq \big((\xi_t-u_t)\lor(\xi_{t^+}-u_{t^+})\big)^+\land(\xi_t-l_t)
\end{split}
\end{equation*}
On the other hand, using the fact that $l\leq u$ and by  definition of $t_i$, we get
\begin{equation*}
\inf_{s_i\leq s\leq t}(\xi_s-l_s)\geq \min(m^i(t^-), \xi_t-l_t)
\end{equation*}
then
$n^i_t=(\xi_t-l_t)$, hence $x_t=l_t$.
\item[$\bullet$] If $t\in]t_i,s_{i+1}[$ for same $i$, there is only the case when $\alpha(y)_{t^-}>\alpha(y)_{s}$ for every $s>t$. The proof  follows similarly as the previous case.
\end{itemize}
This complete the proof.

By the same method and similar computation, we can  show  \eqref{d2} and 
$$\int_{[0,+\infty[} \big((u_s-x_{s})\land (u_{s^+}-x_{s^{+}})\big)\rm{d}\phi^{2,r}_{s}= 0.
$$
	\end{itemize}
\end{itemize}

\begin{proposition}
	Let $l$ and $u$ be two elements of $\mathcal{R}(\mathbb{R}^+,\mathbb{R})$, such that $l_0\leq y_0\leq u{_0}$, $l\leq u$, and for all $t\geq 0$,  $\inf\limits_{s\leq t}(u_s-l_s)>0$. Let $\Gamma_l$ and  $\Gamma_l^u$ be the Skorokhod maps in the time-dependent interval $[l,+\infty[$, and $[l,u]$ respectively. We consider the following  mapping 	
\begin{align*}
\Lambda:\mathcal{R}(\mathbb{R}^+,\mathbb{R})&\rightarrow \mathcal{R}(\mathbb{R}^+,\mathbb{R})\\
f&\mapsto f-\Theta^u_l(f)
\end{align*} 
then, we  have
$$\Gamma_l^u=\Lambda\circ\Gamma_l.$$
\end{proposition}
\begin{proof}
The proof follows immediately from Theorem \ref{th1}, and Lemma \ref{prop3}.
\end{proof}
\begin{remark}
\begin{itemize}
	\item[i)]Note that the definition and solution of the reflection problem $RP_l^u$ coincide  with the classical reflection problem when the driven process $y$ and the barriers are right continuous left limited.
		\item[ii)] The mapping $\Lambda$ coincide with the mapping $\Lambda_a$ defined in \cite{kruk} on the space of right continuous functions with left limits and the case when $l$ and $u$  are constant $(l=0\;\text{ and }\; u=a)$. 
\end{itemize}
\end{remark}
\begin{remark}\label{remark}

$k$ has the following explicit expressions:
	\begin{equation*}
	k_t=\max\Big(\min\big(k_{t_-},((u_t-y_t)\land(u_{t_+}-y_{t_+}))\lor(l_t-y_t)\big), ((l_t-y_t)\lor(l_{t_+}-y_{t_+}))\land(u_t-y_t)\Big),
	\end{equation*}
	and
	\begin{equation*}
	k_{t_+}=\max\Big(\min\big(k_t,u_{t_+}-y_{t_+}),l_{t_+}-y_{t_+}\big)\Big).
	\end{equation*}
	Indeed,  if $\Delta^-k_t<0$ and $\Delta^+k_t=0$ then from the statement (iv) in Definition \ref{def1}, $k_t=(u_t-y_t)\land(u_{t_+}-y_{t_+})$. If 
	$\Delta^-k_t<0$ and $\Delta^+k_t\neq 0$, then $k_t=u_t-y_t$ if $\Delta^+k_t>0$ and $k_t=l_t-y_t$ otherwise. This implies 
	$$k_t=\min(k_{t_-},\big((u_t-y_t)\land(u_{t_+}-y_{t_+}))\lor(l_t-y_t)\big)).$$
	 If $\Delta^-k_t>0$ and $\Delta^+k_t=0$ then from the statement (iv) in Definition \ref{def1}, $k_t=(l_t-y_t)\lor(l_{t_+}-y_{t_+})$. If 
	$\Delta^-k_t>0$ and $\Delta^+k_t\neq 0$, then $k_t=u_t-y_t$ if $\Delta^+k_t>0$ and $k_t=l_t-y_t$ otherwise. This implies $k_t=((l_t-y_t)\lor(l_{t_+}-y_{t_+}))\land(u_t-y_t)\geq k_{t_-}$.
	Hence the expression of $k$ follows. 
	
	From the statement (v) in Definition \ref{def1},we get  the expression of $k_+$.
\end{remark}
\begin{proposition}\label{prop2}
	Let $y$, $\tilde{y}$, $l$, $\tilde{l}$, $u$, and $\tilde{u}$ be in $\mathcal{R}(\mathbb{R}^{+},\mathbb{R})$. Let $(x,k)$ and $(\tilde{x},\tilde{k})$ be the solutions associated to $RP^u_l(y)$ and $RP^{\tilde{u}}_{\tilde{l}}(\tilde{y})$, respectively. Let $T>0$, and denote by $||.||_{<T}$ the supremum norm on the interval $[0,T[$. Then, we have :
$$||k-\tilde{k}||_{<T}\leq 2(||y-\tilde{y}||_{<T}+||l-\tilde{l}||_{<T}+||u-\tilde{u}||_{<T}),$$
$\;\;\text{and}\;\;$
$$||x-\tilde{x}||_{<T}\leq 3(||y-\tilde{y}||_{<T}+||l-\tilde{l}||_{<T}+||u-\tilde{u}||_{<T}).$$
\end{proposition}
\begin{proof}
	Let $(\xi,\kappa)=RP_l(y)$, and $(\tilde{\xi},\tilde{\kappa})=RP_{\tilde{l}}(\tilde{y})$. Let $t\geq 0$, from Lemma \ref{prop3}, we have
\begin{equation*}
k_t=-\bigg(\alpha^{l}(y)_t + \Theta^u_l(\xi)_t\bigg),\;\text{ and }\; \tilde{k}_t=-\bigg(\alpha^{\tilde{l}}(\tilde{y})_t + \Theta^{\tilde{u}}_{\tilde{l}}(\tilde{\xi})_t\bigg).
\end{equation*}

Using the formulas $|p\land q- p'\land q'|\leq|p-p'|\lor|q-q'| $, and $|p\lor q- p'\lor q'|\leq|p-p'|\lor|q-q'|$,  we obtain
	\begin{equation*}
	\begin{split}
	|\Theta^u_l(\xi)_t- \Theta^{\tilde{u}}_{\tilde{l}}(\tilde{\xi})_t|
	&\leq \Bigg( \sup_{s\leq t}\Big(\xi_s-u_s-\tilde{\xi}_s+\tilde{u}_s \Big)\lor \sup_{s\leq t}\Big(\xi_{s^+}-u_{s^+}-\tilde{\xi}_{s^+}+\tilde{u}_{s^+} \Big)\lor\sup_{s\leq t}\Big(\xi_s-l_s-\tilde{\xi}_s+\tilde{l}_s \Big)\\
	&\lor \sup_{s\leq t}\Big(\xi_{s^+}-l_{s^+}-\tilde{\xi}_{s^+}+\tilde{l}_{s^+} \Big) \Bigg).
	\end{split}
	\end{equation*}
	Using the fact that $$\sup_{s<T }\Big(\xi_{s^+}-u_{s^+}-\tilde{\xi}_{s^+}+\tilde{u}_{s^+} \Big)\leq\sup_{s<T}\Big(\xi_s-u_s-\tilde{\xi}_s+\tilde{u}_s \Big) ,$$
	and that
	$$ \sup_{s<T }\Big(\xi_{s^+}-l_{s^+}-\tilde{\xi}_{s^+}+\tilde{l}_{s^+} \Big) \leq \sup_{s<T}\Big(\xi_s-l_s-\tilde{\xi}_s+\tilde{l}_s \Big) $$
	we get
	$$||\Theta^u_l(\xi)- \Theta^u_l(\tilde{\xi})||_{<T}\leq ||\xi-\tilde{\xi}||_{<T}+||l-\tilde{l}||_{<T}+||u-\tilde{u}||_{<T}.$$
	Using the same arguments, we see that
		$$||\alpha^l(y)- \alpha^l(\tilde{y})||_{<T}\leq ||y-\tilde{y}||_{<T}+||l-\tilde{l}||_{<T}+||u-\tilde{u}||_{<T},$$
	 since $\xi=y-\alpha^l(y)$ and $x=y+k$, we obtain
	$$||k-\tilde{k}||_{<T}\leq 2(||y-\tilde{y}||_{<T}+||l-\tilde{l}||_{<T}+||u-\tilde{u}||_{<T},$$
and
	$$||x-\tilde{x}||_{<T}\leq 3(||y-\tilde{y}||_{<T}+||l-\tilde{l}||_{<T}+||u-\tilde{u}||_{<T}.$$
\end{proof}
\begin{remark}
Let $l$ and ,$u$ be two elements of $\mathcal{R}(\mathbb{R}^{+},\mathbb{R})$ such that $l\leq u$ then for all $t\geq 0$, we have
\begin{equation*}
\inf\limits_{s\leq t}(u_s-l_s)>0 \Longleftrightarrow l_{t^-}<l_{t^-}, l_t<u_t, \text{ and }\, l_{t^+}<u_{t^+}.
\end{equation*}

Indeed, The direct implication is obvious. Conversely, let $t\geq 0$ and  assume that
$l_{t^-}<u_{t^-}$, $l_t<u_t$, and $l_{t^+}<u_{t^+}$.
By the sequential characterization of the infimum, we can find a sequence $(s_{n})_{n\geq 0}$ in $[0,t]$ such that, $$\lim\limits_{
	n\rightarrow+\infty}(u_{s_{n}}-l_{s_{n}})=\inf_{s\in [0,t]}(u_{t}-l_{t}).$$
 There exists a  sub-sequence $(s_{n_{k}})_{k\geq 0}$  converging to $s\in [0,t]$. Since either one of the three sets, 
$\{k: s_{n_{k}}>s\}$,  $\{k: s_{n_{k}}=s\}$, or $\{k: s_{n_{k}}<s\}$,  is infinite
we obtain,
$$\lim\limits_{k\rightarrow+\infty}(u_{s_{n_{k}}}-l_{s_{n_{k}}})=  u_{s}-l_{s}\;\text{or}\; \lim\limits_{k\rightarrow+\infty}(u_{s_{n_{k}}}-l_{s_{n_{k}}})=  u_{s^{-}}-l_{s^{-}},\;\text{or}\; \lim\limits_{k\rightarrow+\infty}(u_{s_{n_{k}}}-l_{s_{n_{k}}})=  u_{s^{+}}-l_{s^{+}}.$$
which implies that 
$\inf_{t\in [0,T]}(u_{t}-l_{t})>0.$

\end{remark}

\section{SDEs driven by optional semimartingales and two reflecting barriers}\label{sec4}
 Throughout  this Section, 
 we consider a complete probability space $(\Omega,\mathcal{F},\mathbb{P})$ equipped with a filtration $\mathbb{F}=(\mathcal{F}_{t})_{t\geq 0}$ which is not necessarily right or left continuous and it is not necessary  completed.

 We define the  problem of reflection  for processes as follows:	
  
\begin{definition}
	Let $Y$, $L$, and $U$ be three optional processes with regulated trajectories, such that $L\leq U$, and $L_0\leq Y_0\leq U_0$. We say that a pair of  processes $(X,K)$ is a solution of the reflection problem associated with $Y$ and the barriers $L$ and $U$ if $X$ and $K$ are  $\mathcal{O}(\mathbb{F}^\mathbb{P})$ measurable, and  $(X,K)$ is a solution of $RP^U_L(Y)$. 
	\end{definition} 
\begin{proposition}\label{p}
	Let $Y$ be a semimartingale, $L$ and $U$ be predictable  processes with regulated trajectories such that  $\inf\limits_{s\leq t}(U_s-L_s)>0$ for every $t\geq 0$ and $L_0\leq Y_0\leq U_0$. There exists an explicit unique solution to the reflection problem $RP^U_L(Y)$, which is given by
	$$K=-\Big(\alpha^{L}(Y)\lor \beta^{U}_L(Y)\Big)=-\Big(\alpha^{L}(Y) + \Theta^u_l(\xi)\Big),\; X=Y+K,$$
	where $\xi$ is the first coordinate  of the reflection problem $RP_L(Y)$.
\end{proposition}
\begin{proof}
It follows from Theorems \ref{uniqueness} and \ref{th1} that  for every $\omega\in\Omega$, there exists a unique solution $(X(\omega),K(\omega))$ to the reflection problem $RP^{U(\omega)}_{L(\omega)}(Y(\omega))$ such that
$$K(\omega)=-\Big(\alpha^{L(\omega)}(Y(\omega))\lor \beta^{U(\omega)}_{L(\omega)}(Y(\omega))\Big)=-\Big(\alpha^{L(\omega)}(Y(\omega))+ \Theta^{U(\omega)}_{L(\omega)}(\xi(\omega))\Big), $$
and 
$$X(\omega)=Y(\omega)+K(\omega).$$
Let us show that $X$ and $K$ are $\mathcal{O}(\mathbb{F}^{\mathbb{P}})$-measurable. Note that the process $\alpha^L(Y)$ can be rewritten as
  $$\alpha^L(Y)_t=\alpha^L(Y)_{t^-}\land(Y_t-L_t) \land(Y_{t^+}-L_{t^+}).$$
   Since $Y^g_+$, $L_{+}$, and $U_{+}$ are $\mathcal{P}(\mathbb{F}_+^{\mathbb{P}})$-measurable, from Theorem 1.5 in \cite{galchuk}, there exists $\overline{Y^g}$, $\overline{L}$, and $\overline{U}$  that are $\mathcal{P}(\mathbb{F})$ measurable which are indistinguishable from  $Y^g_+$, $L_{+}$, and $U_{+}$, respectively.
  Therefore, there exists an $\mathbb{F}$ -adapted process $\overline{\alpha}$ which is indistinguishable from $\alpha^L(Y)$. Similarly, we can find an  $\mathbb{F}$-adapted process indistinguishable from $\beta^U_L(Y)$. Hence, $K$ is  $\mathbb{F}^{\mathbb{P}}$-adapted process of bounded variation, therefore  $X$ and $K$ are  $\mathcal{O}(\mathbb{F}^{\mathbb{P}})$-measurable.
\end{proof}
\begin{remark}
Note that if the processes $Y$, $L$ and $U$ are optional  processes with regulated trajectories, and if $RP^U_L(Y)=(X,K)$, then $X$ and $K$ are $\mathcal{O}(\mathbb{F}_+^{\mathbb{P}})$-measurable.
\end{remark}
In what follows,  $L$ and $U$ are predictable processes with regulated trajectories  such that  $\inf\limits_{s\leq t}(U_s-L_s)>0$, $t\geq 0$. 
\begin{lemma}\label{lemma2}
	Let $Y$ and $\tilde{Y}$ be two semimartingales with the the following decomposition: $Y=M+V$ and $\tilde{Y}=\tilde{M}+\tilde{V}$  such that $L_0\leq Y_{0} \leq U_0$ and $L_0\leq \tilde{Y}_{0}\leq U_0$. Assume that $M$ and $\tilde{M}$
	are  in $\mathcal{M}^2$ and $V$ and $\tilde{V}$ are optional processes of bounded variation such that $|V|$ and $|\tilde{V}|$ are square integrable. Let $(X,K)$  and  $(\tilde{X},\tilde{K})$ be the solutions associated to $RP^U_L(Y)$ and $RP^U_L(\tilde{Y})$, respectively. Then, there exists a constant $C$ such that for all stopping times $\tau$  we have:
	\begin{multline*}
	\mathbb{E}(\operatorname*{ess\,sup}_{T\in\mathcal{T}_{<\tau}} |K_{T}-\tilde{K}_{T}|^{2})+\mathbb{E}(\operatorname*{ess\,sup}_{T\in\mathcal{T}_{<\tau}}  |X_{T}-\tilde{X}_{T}|^{2} )\leq\\  C \Bigg[\mathbb{E}([M-\tilde{M}]_{\tau^{-}}) +\mathbb{E}(<M-\tilde{M}>_{\tau^{-}})+\mathbb{E}\big((|V-\tilde{V}|_{\tau^-})^{2}\big)\Bigg].
	\end{multline*}
	Here $\mathcal{T}_{<\tau}$ denotes the set of all stopping times $T$ such that $\mathcal{P}(T<\tau)=1$, and $<.,.>$ is the strongly predictable increasing process associated with the martingale $M-\tilde{M}$.
\end{lemma}
\begin{proof}
From the expression of $K$, and  by using the same method as in the proof of Proposition
\ref{prop2}, we can show that,
$$ \mathbb{E}(\operatorname*{ess\,sup}_{T\in\mathcal{T}_{<\tau}} |K_{T}-\tilde{K}_{T}|^{2})\leq 3 \mathbb{E}(\operatorname*{ess\,sup}_{T\in\mathcal{T}_{<\tau}} |Y_{T}-\tilde{Y}_{T}|^{2}).\;\;  $$
Consequently,
$$ \mathbb{E}(\operatorname*{ess\,sup}_{T\in\mathcal{T}_{<\tau}} |K_{T}-\tilde{K}_{T}|^{2})\leq 3 \Big( \mathbb{E}(\operatorname*{ess\,sup}_{T\in\mathcal{T}_{<\tau}} |M_{T}-\tilde{M}_{T}|^{2})+ \mathbb{E}(\operatorname*{ess\,sup}_{T\in\mathcal{T}_{<\tau}} |V_{T}-\tilde{V}_{T}|^{2})\Big)  $$
On one hand, we have,
$$\mathbb{E}(\operatorname*{ess\,sup}_{T\in\mathcal{T}_{<\tau}} |M_{T}-\tilde{M}_{T}|^{2})\leq 2( \mathbb{E}(\operatorname*{ess\,sup}_{T\in\mathcal{T}_{<\tau}} |M^r_{T}-\tilde{M}^r_{T}|^{2})+\mathbb{E}(\operatorname*{ess\,sup}_{T\in\mathcal{T}_{<\tau}} |M^g_{T}-\tilde{M}^g_{T}|^{2})).$$
Since $M^r$ is right continuous left limited and $M^g$ is left continuous right limited, we obtain
$$\mathbb{E}(\operatorname*{ess\,sup}_{T\in\mathcal{T}_{<\tau}} |M_{T}-\tilde{M}_{T}|^{2})\leq 2\Big( \mathbb{E}(\sup_{s< \tau} |M^r_{s}-\tilde{M}^r_{s}|^{2})+\mathbb{E}(\sup_{s< \tau} |M^g_{s}-\tilde{M}^g_{s}|^{2})\Big).$$
From Métivier-péllaumail inequality, we have
\begin{equation*}
\mathbb{E}(\sup_{s<\tau}|M^r_{s}-\tilde{M}^r_{s}|^{2})\leq 4 \mathbb{E}\Big(<M^r-\tilde{M^r}>_{\tau^{-}}+[M^r-\tilde{M^r}]_{\tau^{-}}\Big),
\end{equation*}
it follows from Burkholder David Gundy's  inequality \cite{galchuk85}, that there exists $C_1$ such
\begin{equation*}
\mathbb{E}(\sup_{s<\tau}|M^g_{s}-\tilde{M}^g_{s}|^{2})\leq  C_1\mathbb{E}\Big([M^g-\tilde{M^g}]_{\tau}\Big)=C_1\mathbb{E}\Big([M^g-\tilde{M^g}]_{\tau^-}\Big),
\end{equation*}
therefore, with $C=6\max(C_1,4)$, we obtain the desired inequality. 
\end{proof}
\begin{remark}\label{remark1}
From the expression of solution of the reflection problem $RP^U_L$, we  see that  if $(X,K)$ is a solution of $RP_L^U(Y)$  and if for a stopping time $\tau$, $\mathbb{E}(\operatorname*{ess\,sup}_{T\in\mathcal{T}_{< \tau}} (|Y_{T}|^2+|L_{T}|^2+|U_{T}|^2))<\infty$,  then 
$$\mathbb{E}(\operatorname*{ess\,sup}_{T\in\mathcal{T}_{< \tau}} |K_{T}|^2)\leq2\big(\mathbb{E}(\operatorname*{ess\,sup}_{T\in\mathcal{T}_{< \tau}} |Y_{T}|^2)+\mathbb{E}(\operatorname*{ess\,sup}_{T\in\mathcal{T}_{< \tau}} |L_{T}|^2)+\mathbb{E}(\operatorname*{ess\,sup}_{T\in\mathcal{T}_{< \tau}} |U_{T}|^2)\big).$$
\end{remark}
\begin{definition}
	Let $X_{0}$  be  $\mathcal{F}_{0}$ measurable  such that $L_0\leq X_{0} \leq U_0$, and that $L_0$ and $U_0$ are bounded. Let $\sigma$ and $b$ be two random functions defined on $\mathbb{R}^+\times\Omega\times\mathbb{R}$. Let $M$ be an optional local martingale and $V$ be an  optional process of bounded variation. A couple $(X, K)$ is said to be a solution to the reflected SDE, that we denote $E(\sigma, b,L,U)$, if
\begin{itemize}
	\item[(i)] $(X_{t})_{ t\geq 0}$ is $\mathcal{O}(\mathbb{F}^{\mathbb{P}})$-measurable with regulated trajectories; 
	\item[(ii)] $K$ is $\mathbb{F}^{\mathbb{P}}$-adapted and has a bounded variation with $K_{0}=0$;
	\item[(iii)] \begin{equation*} L\leq X\leq U  \end{equation*}
	\item[(iv)]
	\begin{equation*}
	\int_{[0,+\infty[} (X_{s}-L_s)\rm{d}K^{c}_{s}\leq 0,\;\;\int_{[0,+\infty[} (U_s-X_{s})\rm{d}K^{c}_{s}\geq 0.
	\end{equation*}	
	\begin{equation*}
\sum_{s\leq t} \big((X_{s}-L_s)\land (X_{s^{+}}-L_{s^+})\big)\Delta^{-}K_s\leq 0,\;\; \sum_{s\leq t} \big((U_s-X_{s})\land (U_{s^+}-X_{s^{+}})\big)\Delta^{-}K_s\geq 0
	\end{equation*}
	\item[(v)] 	
	\begin{equation*}
	\sum_{s\leq t}(X_{s^{+}}-L_{s^+})\Delta^{+}K_s\leq 0\;\text{ and }\;\sum_{s\leq t}(u_{s}-x_{s})\Delta^{+}K_s\leq 0.
	\end{equation*}
	\begin{equation*}
	\sum_{s\leq t}(U_{s^+}-X_{s^{+}})\Delta^{+}K_s\geq 0\;\text{ and }\; \sum_{s\leq t}(X_{s}-L_{s})\Delta^{+}K_s\geq0.
	\end{equation*}
\item[(vi)]
	\begin{equation}
	X =X_0+\sigma(.,X)\bigcdot M+b(.,X)\bigcdot V+K . \label{I17}
	\end{equation}
\end{itemize}
\end{definition}
\begin{theorem}
 We assume that $\sigma$ and $b$ satisfy the following,
\begin{itemize}
\item[(i)] There exists $\lambda>0$ such that $x\to \sigma(.,.,x)$, and $x\to b(.,.,x)$ are $\lambda$-Lipschitz.
\item[(ii)] $(t,\omega)\to \sigma(t,\omega,.)$ and $(t,\omega)\to b(t,\omega,.)$ are $\mathbb{F}$-adapted processes with regulated trajectories. 
\end{itemize}
 Then there exists a unique solution of the reflected stochastic differential equation $E(\sigma,b,L,U)$.
\end{theorem}
\begin{proof}
From (i) and (ii) the processes    $(\sigma(t_{-},X_{t^-})_{t\geq 0})$ and $(b(t_{-},X_{t^-})_{t\geq 0})$ are $\mathcal{P}(\mathbb{F})$ measurable. Hence the integrals in \eqref{I17} are well defined.

Without loss of generality, we may assume that $M^r$ has a bounded  jumps (Proposition \ref{r4}). Since $L_0$ and $U_0$ are bounded, we may also assume that there exists a constant $c$ such that $|L|+|U|<c$ (Lemma \ref{lemmaA}).
Let $m$ such that $0<3C\lambda^2m<1$, and set 
$$\tau:=\inf\{t>0 \;:\; [M]_{t^+}+(|V|_{t^+})^2+(\sigma(.,0)^2\bigcdot [M]_{t^+})+(b(.,0)^2\bigcdot |V|_{t^+})\geq m \} \;\text{and}\; \tau=\infty\; \text{for} \;\inf(\emptyset).$$
Since by Theorem 1.5  in \cite{galchuk}, there exists and $\mathcal{P}(\mathbb{F})$  indistinguishable processes from $[M^g]_+$ and  $|V^g|$,  then $\tau $ is an $\mathbb{F}^{\mathbb{P}}$ stopping time. 

We denote by  $\mathcal{S}^{2}$ the complete space of real valued    $\mathcal{O}(\mathbb{F}^{\mathbb{P}})$-measurable and regulated processes $(X_{t})_{t\geq 0}$ such that   $||X||_{\mathcal{S}^{2}}=|| \operatorname*{ess\,sup}_{T\in\mathcal{T}_{< \tau}} |X_{T}|\; ||_{\mathbb{L}^{2}}<+\infty$.

We consider the mapping $\varphi \colon\mathcal{S}^{2}\to\mathcal{S}^{2} $ that associates $X\in\mathcal{S}^{2}$ to $\varphi(X) $, where $\varphi(X) $ is defined as the first coordinate of the solution of the reflection problem: $RP^U_L\Big(X_0+\sigma(.,X) \bigcdot M^{\tau^-}+b(.,X)\bigcdot V^{\tau^-} \Big).$

 From Lemma \ref{lemma2} and the condition (i), we get: 
 $$|| \varphi(X)-\varphi(\tilde{X}) ||_{\mathcal{S}^{2}}^{2}\leq 3C\lambda^2m || X-\tilde{X} ||_{\mathcal{S}^{2}}^{2},$$
   by definition of $\tau$ and from Remark \ref{remark1}, $\varphi(0)\in\mathcal{S}^2$, and since $\varphi$ is Lipschitz,  we get that for every $X\in\mathcal{S}^{2}$, $\varphi(X)\in\mathcal{S}^{2}$. 
By the Banach fixed point Theorem, there exists a unique process $X$ in $\mathcal{S}^{2}$ such that $\varphi(X)=X$.
This implies that  $E(\sigma,b,L,U)$ has a unique solution on $[0,\tau[$.
From Remark \ref{remark}, we can extend the solution to $[0,\tau]$ by setting: 
\begin{equation*}
K_{\tau}=\max\Big(\min\big(K_{\tau_-},((U_{\tau}-Y_{\tau})\land(U_{\tau_+}-Y_{\tau_+}))\lor(L_{\tau}-Y_{\tau})\big), ((L_{\tau}-Y_{\tau})\lor(L_{\tau_+}-Y_{\tau_+}))\land(U_\tau-Y_\tau)\Big),
\end{equation*}
\begin{equation*}
K_{\tau_+}=\max\Big(\min\big(K_\tau,U_{\tau_+}-Y_{\tau_+}),L_{\tau_+}-Y_{\tau_+}\big)\Big),
\end{equation*} 
$$X_{\tau}=Y_{\tau}+K_{\tau},$$
where $Y=X_0+\sigma(.,X) \bigcdot M+b(.,X)\bigcdot V$.
Therefore $(X,K)$ is a solution of $E(\sigma,b,L,U)$ on $[0,\tau]$.
By induction, we define the following sequence of stopping times:

$\tau_{0}=0,$ and $$\tau_{n+1}:=\inf\{t>\tau_{n}\;:\; ([M]_{t^+}-[M]_{\tau_{n}^+})+(|V|_{t^+}-|V|_{\tau_n^+})^{2}  + \sigma(.,0)^2\bigcdot ([M]_{t^+}-[M]_{\tau_n^+})+ b(.,0)^2\bigcdot (|V|_{t^+}-|V|_{\tau_n^+}) \geq m \},$$

with $\inf(\emptyset)=\infty$.

By a similar argument as before, one can show the existence and uniqueness of the solution on each interval $[\tau_{n},\tau_{n+1}]$. Since  $t\to [M]_t+\int_{[0,t]}\rm{d}|V|_t+\sigma(.,0)^2\bigcdot M_t+b(.,0)^2\bigcdot |V|_t$ has regulated trajectories, $\lim\limits_{n\to+\infty} \tau_{n}=+\infty$. Consequently, $E(\sigma,b,L,U)$ has a unique solution on $\mathbb{R}^{+}$.
\end{proof}
\appendix
\section{Appendix}
\begin{proposition}[\cite{hilbert}]\label{r1}
	Let $y$ and $l$ be in  $\mathcal{R}(\mathbb{R}^{+},\mathbb{R})$ such that $y_{0}\geq l_0$, then there exists a unique couple of  functions $(\xi,\kappa)$ satisfy the following:
	\begin{itemize}
		\item[i)]$\xi=y+\kappa\geq l$
		\item[ii)]  $\kappa$ is increasing, right continuous and $\kappa_{0}=0$, 
		\item[iii)] 
			\begin{equation*}
		\int_{[0,+\infty[} (\xi_{s}-l_s)\land (\xi_{s^{+}}-l_{s^+})d\kappa_{s}=0.
		\end{equation*}
	\end{itemize}
We denote this by: $(x,k)=RP_l(y)$.

In addition, $\kappa$ has an explicit expression that is given by:
	\begin{equation*}
	\kappa_{t}=-\alpha^{l}(y)_t.\label{A1}
	\end{equation*} 
\end{proposition}
\begin{proposition}\label{r2}
Let $y$ and $u$ be in  $\mathcal{R}(\mathbb{R}^{+},\mathbb{R})$ such that $y_{0}\leq u_0$, then there exists a unique couple of  functions $(\xi,\kappa)$ satisfy the following:
\begin{itemize}
	\item[i)]$\xi=y-\kappa\leq u$
	\item[ii)]  $\kappa$ is increasing, right continuous and $\kappa_{0}=0$, 
	\item[iii)] 
	\begin{equation*}
	\int_{[0,+\infty[} (u_s-\xi_{s})\land (u_{s^{+}}-\xi_{s^{+}})d\kappa_{s}=0.
	\end{equation*}
\end{itemize}
we denote this by: $(x,k)=RP^u(y)$.

In addition, $\kappa$ has an explicit expression that is given by:
\begin{equation*}
\kappa_{t}=-\alpha^{y}(u)_t.\label{A2}
\end{equation*} 
\end{proposition}
\begin{proposition}\label{r3}
	Let $A$ be a process of  bounded variation, then we have:
	\begin{equation*}
	A_t^2\leq A_0^2+ \int_{]0,t]}(A_s+A_{s^+})\rm{d}A^r_{s} +2\sum_{0\leq s<t} A_{s^+}\Delta^+A_s-\sum_{0< s\leq t}\Delta^+A_s\Delta^-A_s
	\end{equation*} 
\end{proposition}
\begin{proof}
	By Theorem 8.2 in \cite{galchuk}, we have 
	\begin{equation*}
	\begin{split}
	A_t^2&=A_0^2+2\int_{]0,t]}A_{s_-}\rm{d}A^r_s+2\int_{[0,t[}A_{s}\rm{d}A^g_{s_+}+\sum_{0<s\leq t}(\Delta^-A_s)^2+\sum_{0<s< t}(\Delta^+A_s)^2\\
	\\
	&=A_0^2+\int_{]0,t]}(A_s+A_{s_+})\rm{d}A^r_s- \int_{]0,t]}\big(2\Delta^-A_{s}+\Delta^+A_{s}\big)\rm{d}A^r_s+2\int_{[0,t[}A_{s_+}\rm{d}A^g_{s_+}\\
	&-2\sum_{0\leq s<t}(\Delta^+A_s)^2+\sum_{0<s\leq t}(\Delta^-A_s)^2+\sum_{0\leq s< t}(\Delta^+A_s)^2\\
	&=A_0^2+\int_{]0,t]}(A_s+A_{s_+})\rm{d}A^r_s-\sum_{0< s\leq t}\Delta^+A_s\Delta^-A_s +2\sum_{0\leq s< t}A_{s^+}\Delta
	^+A_s-\sum_{0<s\leq t}(\Delta^-A_s)^2-\sum_{0\leq s< t}(\Delta^+A_s)^2\\
	&\leq A_0^2+\int_{]0,t]}(A_s+A_{s_+})\rm{d}A^r_s-\sum_{0< s\leq t}\Delta^+A_s\Delta^-A_s +2\sum_{0\leq s< t}A_{s^+}\Delta
	^+A_s.
	\end{split}
	\end{equation*}
\end{proof}
Fundamental Theorem of Local martingale is still available for right continuous local martingale in  an \textit{un}usual probability space .
\begin{proposition}\label{r4}
Let $M$ be a right continuous local martingale, and $c>0$. Then there exists $\mathcal{O}(\mathbb{F})$ local martingale $N$, and a bounded variation process $D$ such that the jump of $N$ are bounded by $2c$, and $M=N+D$. 
\end{proposition}
\begin{proof}
If we make use of theorem 5.3.8 in \cite{Abdelghani2}, the proof follows similarly as the proof of Theorem 25, Chapitre III in \cite{protter}.
\end{proof}
\begin{lemma}\label{lemmaA}
Let $X$ be a predictable  process with regulated trajectories. Then  there exists a sequence of $\mathbb{F}^{\mathbb{P}}_+$-stopping times $\tau_n$, such that $X^{\tau_n}$ is bounded.
\end{lemma}

\begin{proof}
Let $k\geq1$, and $R_k=\inf\{t\geq 0:\; |X_{t_+}|\lor|X_t|\geq k\}$, $(R_k)_{k\geq 1}$ is a sequence  $\mathbb{F}^\mathbb{P}_+$ predictable stopping times, then there exists a sequence of  increasing $\mathbb{F}^\mathbb{P}_+$ stopping times $(R_{n}^k)_{n\geq 1}$ such that $\lim\limits_{n\to+\infty}R_{n}^k=R_k$ and $R_n^k<R_k$ on $\{R_k>0\}$. We consider $\tau_n=R_n^1\lor R_n^2\lor...\lor R_n^n$.
Since $\lim\limits_{k}R_k=+\infty$ then $\tau_n$ increase to the infinity. Since $\tau_n<R_n$ on $\{\tau_n>0\}$, we have for every $n$, $$|X^{\tau_n}|=|X_0|1_{\{\tau_n=0\}}+|X^{\tau_n}|1_{\{\tau_n>0\}}\leq |X_0|+n.$$
This completes the proof. 
\end{proof}
	
\end{document}